\documentclass[11pt]{amsart}
\usepackage[margin=1in]{geometry}
\usepackage{latexsym}
\usepackage{amsfonts}
\usepackage{amsmath}
\usepackage{amssymb}
\usepackage{amsthm}
\usepackage{enumerate}
\usepackage[hang,flushmargin]{footmisc}
\usepackage{caption}
\usepackage{tabu}
\usepackage{mathrsfs}
\usepackage{amsaddr}
\usepackage{subfig}
\usepackage{makecell}

\newcommand{\Av}{\operatorname{Av}}

\usepackage{graphicx}
\usepackage{epstopdf}
\usepackage{epsfig}
\usepackage{caption}
 
\usepackage{bm} 
\usepackage{cite}

\makeatletter
\newtheorem*{rep@theorem}{\rep@title}
\newcommand{\newreptheorem}[2]{%
\newenvironment{rep#1}[1]{%
 \def\rep@title{#2 \ref{##1}}%
 \begin{rep@theorem}}%
 {\end{rep@theorem}}}
\makeatother

\newtheorem{theorem}{Theorem}[section]
\newreptheorem{theorem}{Conjecture}
\newtheorem{lemma}{Lemma}[section]
\newtheorem{proposition}{Proposition}[section]
\newtheorem{corollary}{Corollary}[section]

\theoremstyle{definition}
\newtheorem{definition}{Definition}[section]
\newtheorem{remark}{Remark}[section]

\DeclareMathOperator{\DL}{\Lambda\hspace{-.095cm}\Lambda}
\DeclareMathOperator{\swd}{swd}
\DeclareMathOperator{\swu}{swu}
\DeclareMathOperator{\swl}{swl}
\DeclareMathOperator{\swr}{swr}
\DeclareMathOperator{\lon}{long}
\DeclareMathOperator{\NC}{NC}
\DeclareMathOperator{\Int}{Int}

\DeclareMathOperator{\rot}{rot}
\DeclareMathOperator{\rev}{rev}
\DeclareMathOperator{\des}{des}
\DeclareMathOperator{\Des}{Des}
\DeclareMathOperator{\de}{def}

\DeclareMathOperator{\DB}{DB}

\DeclareMathOperator{\spli}{split}
\DeclareMathOperator{\PC}{PC}

\begin{document}
\title{Catalan Intervals and Uniquely Sorted Permutations}
\author{Colin Defant}
\address{Princeton University \\ Fine Hall, 304 Washington Rd. \\ Princeton, NJ 08544}
\email{cdefant@princeton.edu}

\begin{abstract}
For each positive integer $k$, we consider five well-studied posets defined on the set of Dyck paths of semilength $k$. We prove that uniquely sorted permutations avoiding various patterns are equinumerous with intervals in these posets. While most of our proofs are bijective, some use generating trees and generating functions. We end with several conjectures.
\end{abstract}

\maketitle

\bigskip

\section{Introduction}

A \emph{Dyck path of semilength} $k$ is a lattice path in the plane consisting of $k$ $(1,1)$ steps (also called \emph{up steps}) and $k$ $(1,-1)$ steps (also called \emph{down steps}) that starts at the origin and never passes below the horizontal axis. Letting $U$ and $D$ denote up steps and down steps, respectively, we can view a Dyck path of semilength $k$ as a word over the alphabet $\{U,D\}$ that contains $k$ copies of each letter and has the property that every prefix has at least as many $U$'s as it has $D$'s. The number of such paths is the $k^\text{th}$ Catalan number $C_k=\frac{1}{k+1}{2k\choose k}$; this is just one of the overwhelmingly abundant incarnations of these numbers. 
\begin{figure}[h]
\begin{center}
\includegraphics[width=.25\linewidth]{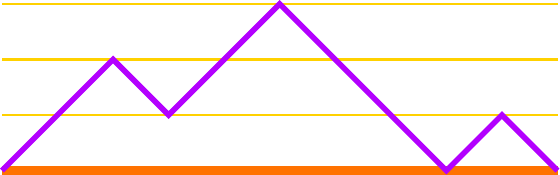}
\caption{The Dyck path $UUDUUDDDUD$ of semilength $5$.}
\label{Fig1}
\end{center}  
\end{figure}

Let ${\bf D}_k$ be the set of Dyck paths of semilength $k$. We obtain a natural partial order $\leq_S$ on ${\bf D}_k$ by declaring that $\Lambda\leq_S\Lambda'$ if $\Lambda$ lies weakly below $\Lambda'$. Alternatively, we have $\Lambda_1\cdots\Lambda_{2k}\leq_S\Lambda_1'\cdots\Lambda_{2k}'$ if and only if the number of $U$'s in $\Lambda_1\cdots\Lambda_i$ is at most the number of $U$'s in $\Lambda_1'\cdots\Lambda_i'$ for every $i\in\{1,\ldots,2k\}$. The poset $({\bf D}_k,\leq_S)$ turns out to be a distributive lattice; it is known as the $k^\text{th}$ \emph{Stanley lattice} and is denoted by $\mathcal L_k^S$. See \cite{Bernardi, Ferrari1, Ferrari2, Ferrari3} for more information about these fascinating lattices. The upper left image in Figure \ref{Fig2} shows the Hasse diagram of $\mathcal L_3^S$. 

\begin{figure}[h]
\begin{center}
\includegraphics[width=.86\linewidth]{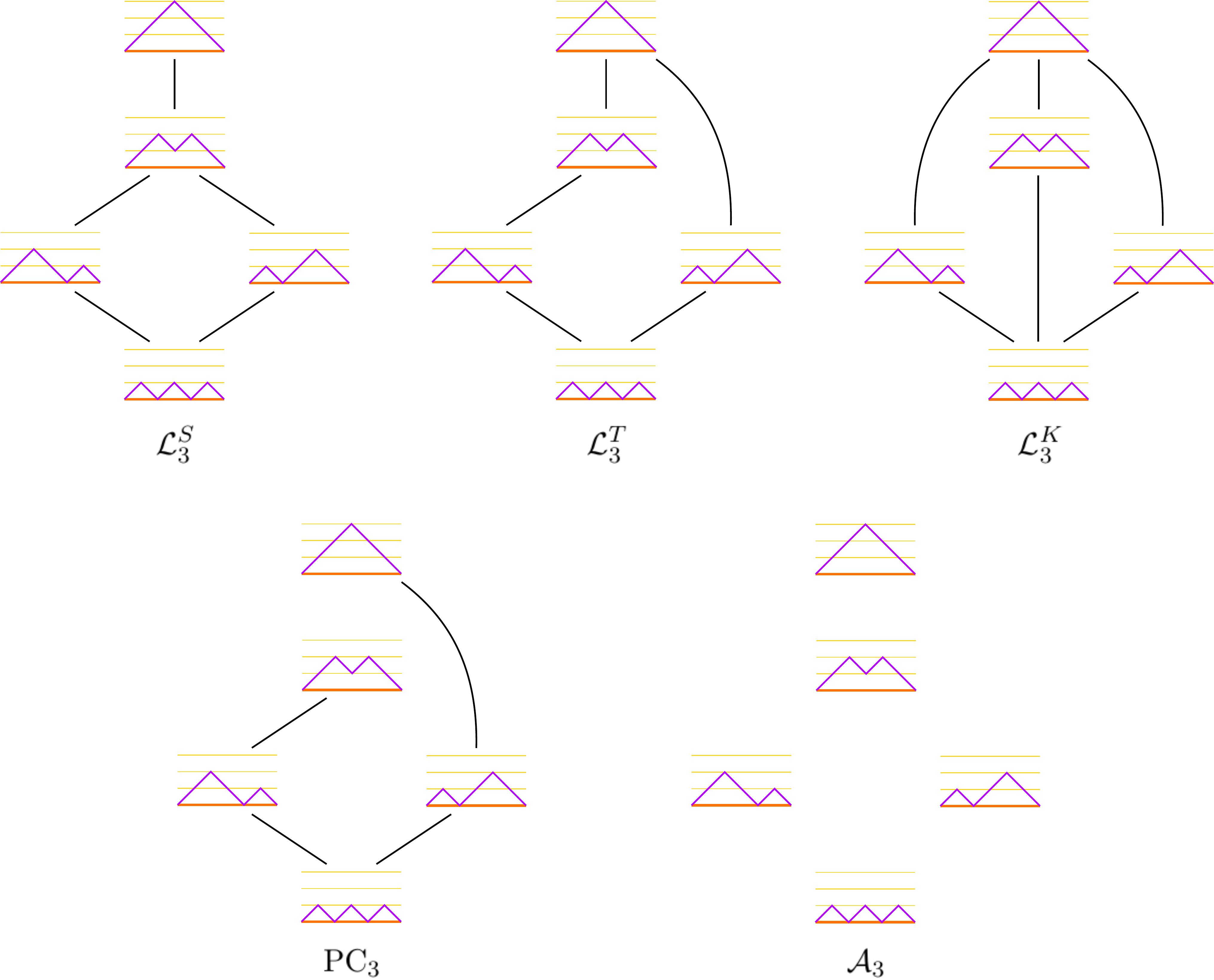}
\caption{The Hasse diagrams of our Catalan posets for $k=3$.}
\label{Fig2}
\end{center}  
\end{figure}

The $k^\text{th}$ \emph{Tamari lattice} is often defined on the set of binary plane trees with $k$ vertices. However, there are several equivalent definitions that allow one to define isomorphic lattices on other sets of objects counted by the $k^\text{th}$ Catalan number. Because the Tamari lattices are so multifaceted, they have been extensively studied in combinatorics and other areas of mathematics \cite{Chapoton, Early, Geyer, Huang, Tamari}. In particular, the Hasse diagrams of Tamari lattices arise as the $1$-skeleta of associahedra \cite{Loday}. If $\leq_1$ and $\leq _2$ are two partial orders on the same set $X$, then we say the poset $(X,\leq_1)$ is an \emph{extension} of the poset $(X,\leq_2)$ if $x\leq_2 y$ implies $x\leq_1 y$ for all $x,y\in X$. Each of the articles \cite{Bernardi,BousquetTamari} described how to define $\mathcal L_k^T$, the $k^\text{th}$ Tamari lattice, so that its underlying set is ${\bf D}_k$; in \cite{Bernardi}, Bernardi and Bonichon proved that the Stanley lattice $\mathcal L_k^S$ is an extension of $\mathcal L_k^T$. We define $\mathcal L_k^T$ in Section \ref{Sec:LatticeBack}. 

In a now-classical paper, Kreweras \cite{Kreweras} investigated the poset $\NC_k$ of all noncrossing partitions of the set $[k]:=\{1,\ldots,k\}$ ordered by refinement, showing, in particular, that this poset is a lattice. It is difficult to overstate the importance and ubiquity of noncrossing partitions and these lattices in mathematics \cite{Adin, Knuth2, McCammond, Simion, Speicher, Stanley}. Using a bijection between Dyck paths and noncrossing partitions, Bernardi and Bonichon \cite{Bernardi} defined an isomorphic copy of $\NC_k$, denoted $\mathcal L_k^K$, so that its underlying set is ${\bf D}_k$. They also showed that $\mathcal L_k^T$ is an extension of $\mathcal L_k^K$. We call $\mathcal L_k^K$ the $k^\text{th}$ \emph{Kreweras lattice}.\footnote{We use the names ``Kreweras lattice" and ``noncrossing partition lattice" to distinguish the underlying sets, even though the lattices themselves are isomorphic.} We will find it more convenient to work with the noncrossing partition lattices instead of the Kreweras lattices, so we refer the interested reader to \cite{Bernardi} for the definition of $\mathcal L_k^K$. 

Bernardi and Bonichon used the name ``Catalan lattices" to refer to $\mathcal L_k^S$, $\mathcal L_k^T$, and $\mathcal L_k^K$. Building off of earlier work of Bonichon \cite{Bonichon}, they gave unified bijections between intervals in these lattices and certain types of triangulations and realizers of triangulations. We find it appropriate to add two additional families of posets to this Catalan clan. The first is the family of \emph{Pallo comb posets}, a relatively new family of posets introduced by Pallo in \cite{Pallo} as natural posets that have the Tamari lattices as extensions. They were studied further in \cite{Aval, Csar}. We let $\PC_k$ denote the $k^\text{th}$ Pallo comb poset. These were defined on sets of binary trees in \cite{Pallo, Csar} and on sets of triangulations in \cite{Aval}; in Section \ref{Sec:LatticeBack}, we define the Pallo comb posets on sets of Dyck paths. The second family of posets we add to the clan is the family of Catalan antichains. That is, we let $\mathcal A_k$ denote the antichain (poset with no nontrivial order relations) defined on the set ${\bf D}_k$. 

An \emph{interval} in a poset $P$ is a pair $(x,y)$ of elements of $P$ such that $x\leq y$. Let $\Int(P)$ be the set of all intervals of $P$. It is often interesting to count the intervals in combinatorial classes of posets, and the Catalan posets defined above are no exceptions. De Sainte-Catherine and Viennot \cite{De} proved that 
\begin{equation}\label{Eq1}
|\Int(\mathcal L_k^S)|=C_kC_{k+2}-C_{k+1}^2=\frac{6}{(k+1)(k+2)^2(k+3)}{2k\choose k}{2k+2\choose k+1}.
\end{equation}
Chapoton \cite{Chapoton} proved that 
\begin{equation}\label{Eq2}
|\Int(\mathcal L_k^T)|=\frac{2}{(3k+1)(3k+2)}{4k+1\choose k+1}.
\end{equation}
In his initial investigation of the noncrossing partition lattices, Kreweras \cite{Kreweras} proved that 
\begin{equation}\label{Eq3}
|\Int(\mathcal L_k^K)|=|\Int(\NC_k)|=\frac{1}{2k+1}{3k\choose k}.
\end{equation}
Aval and Chapoton \cite{Aval} proved that 
\begin{equation}\label{Eq15}
\sum_{k\geq 0}|\Int(\PC_k)|x^k=C(xC(x)),
\end{equation}
where $C(x)=\dfrac{1-\sqrt{1-4x}}{2x}$ is the generating function of the sequence of Catalan numbers. Of course, we also have 
\begin{equation}\label{Eq4}
|\Int(\mathcal A_k)|=|{\bf D}_k|=C_k.
\end{equation}  
The formulas in \eqref{Eq1}, \eqref{Eq2}, \eqref{Eq3}, \eqref{Eq15}, \eqref{Eq4} give rise to the OEIS sequences A005700, A000260, A001764, A127632, A000108, respectively \cite{OEIS}.

Throughout this article, the word ``permutation" refers to an ordering of a set of positive integers, written in one-line notation. Let $S_n$ denote the set of permutations of the set $[n]$. The latter half of the present article's title refers to a special collection of permutations that arise in the study of West's stack-sorting map. This map, denoted by $s$, sends permutations of length $n$ to permutations of length $n$. It is a slight variant of the stack-sorting algorithm that Knuth introduced in \cite{Knuth}. The map $s$ was studied extensively in West's 1990 Ph.D. thesis \cite{West} and has received a considerable amount of attention ever since \cite{Bona, BonaSurvey, Bousquet, DefantCounting}. We give necessary background results concerning the stack-sorting map in Section \ref{Sec:StackBack}, but the reader seeking additional historical motivation should consult \cite{Bona, BonaSurvey, DefantCounting} and the references therein. There are multiple ways to define $s$, but the simplest is probably the following recursive definition. First, $s$ sends the empty permutation to itself. If $\pi$ is a permutation whose largest entry is $n$, then we can write $\pi=LnR$. We then define $s(\pi)=s(L)s(R)n$. For example, \[s(35241)=s(3)\,s(241)\,5=3\,s(2)\,s(1)\,45=32145.\]

One of the central definitions concerning the stack-sorting map is that of the \emph{fertility} of a permutation $\pi$; this is simply $|s^{-1}(\pi)|$, the number of preimages of $\pi$ under $s$. Bousquet-M\'elou called a permutation \emph{sorted} if its fertility is positive. The following much more recent definition appeared first in \cite{DefantEngenMiller}. 
\begin{definition}\label{Def1}
We say a permutation is \emph{uniquely sorted} if its fertility is $1$. Let $\mathcal U_n$ denote the set of uniquely sorted permutations in $S_n$.  
\end{definition}

The following theorem from \cite{DefantEngenMiller} characterizes uniquely sorted permutations. A \emph{descent} of a permutation $\pi=\pi_1\cdots\pi_n$ is an index $i\in[n-1]$ such that $\pi_i>\pi_{i+1}$. We let $\Des(\pi)$ denote the set of descents of the permutation $\pi$ and let $\des(\pi)=|\Des(\pi)|$. 

\begin{theorem}[\!\!\cite{DefantEngenMiller}]\label{Thm1}
A permutation of length $n$ is uniquely sorted if and only if it is sorted and has exactly $\dfrac{n-1}{2}$ descents. 
\end{theorem}

Uniquely sorted permutations contain a large amount of interesting hidden structure. The results in \cite{DefantEngenMiller} hint that, in some loose sense, uniquely sorted permutations are to general sorted permutations what matchings are to general set partitions. For example, one immediate consequence of Theorem~\ref{Thm1} is that there are no uniquely sorted permutations of even length (just as there are no matchings of a set of odd size). The authors of \cite{DefantEngenMiller} defined a bijection between new combinatorial objects called ``valid hook configurations" and certain weighted set partitions that Josuat-Verg\`es \cite{Josuat} studied in the context of free probability theory. They then showed that restricting this bijection to the set of valid hook configurations of uniquely sorted permutations induces a bijection between uniquely sorted permutations and those weighted set partitions that are matchings. This allowed them to prove that $|\mathcal U_{2k+1}|=A_{k+1}$, where $(A_m)_{m\geq 1}$ is OEIS sequence A180874 and is known as \emph{Lassalle's sequence}. This fascinating new sequence first appeared in \cite{Lassalle}, where Lassalle proved a conjecture of Zeilberger by showing that the sequence is increasing. The article \cite{DefantEngenMiller} also proves that the sequences $(A_{k+1}(\ell))_{\ell=1}^{2k+1}$ are symmetric, where $A_{k+1}(\ell)$ is the number of elements of $\mathcal U_{2k+1}$ with first entry $\ell$. 

The present article is meant to link uniquely sorted permutations that avoid certain patterns with intervals in the Catalan posets discussed above. Let $\mathcal U_n(\tau^{(1)},\ldots,\tau^{(r)})$ denote the set of uniquely sorted permutations in $S_n$ that avoid the patterns $\tau^{(1)},\ldots,\tau^{(r)}$ (see the beginning of Section~\ref{Sec:StackBack} for the definition of pattern avoidance). In Section~\ref{Sec:LatticeBack}, we define the Tamari lattices and Pallo comb posets. Section \ref{Sec:StackBack} reviews relevant background concerning the stack-sorting map and permutation patterns. Section \ref{Sec:Operators} introduces new operators that act on permutations. We prove several properties of these operators that are used heavily in the remainder of the paper and in \cite{DefantFertilityWilf}. In Section \ref{Sec:Stanley}, we find a bijection $\mathcal U_{2k+1}(312)\to\Int(\mathcal L_k^S)$, showing that $312$-avoiding uniquely sorted permutations are counted by the numbers in \eqref{Eq1}. The proof that this map is bijective actually relies on a fun ``energy argument" similar to the one used in the solution of the game ``Conway's Soldiers." In Section \ref{Sec:Tamari}, we find bijections $\mathcal U_{2k+1}(231)\to\mathcal U_{2k+1}(132)$ and $\mathcal U_{2k+1}(132)\to\Int(\mathcal L_k^T)$, showing that the permutations in $\mathcal U_{2k+1}(231)$ and the permutations in $\mathcal U_{2k+1}(132)$ are counted by the numbers in \eqref{Eq2}. In Section \ref{Sec:Kreweras}, we use generating trees to exhibit a bijection $\mathcal U_{2k+1}(312,1342)\to\Int(\mathcal L_k^K)$ (which is not a restriction of the aforementioned bijection $\mathcal U_{2k+1}(312)\to\Int(\mathcal L_k^S)$), proving that the permutations in $\mathcal U_{2k+1}(312,1342)$ are counted by the numbers in \eqref{Eq3}. In Section~\ref{Sec:Pallo}, we show that the permutations in $\mathcal U_{2k+1}(231,4132)$ are in bijection with the intervals in $\Int(\PC_k)$ (we do not obtain such a bijection by restricting the aforementioned map $\mathcal U_{2k+1}(231)\to\Int(\mathcal L_k^T)$). In Section~\ref{Sec:Antichain}, we give bijections demonstrating that \[|\mathcal U_{2k+1}(321)|=|\mathcal U_{2k+1}(132,231)|=|\mathcal U_{2k+1}(132,312)|=|\mathcal U_{2k+1}(231,312)|=C_k.\] Thus, these sets of permutations are in bijection with intervals of the antichain $\mathcal A_k$. In fact, many of the maps from other sections restrict to bijections with antichain intervals (for example, the bijection $\mathcal U_{2k+1}(312)\to\Int(\mathcal L_k^S)$ from Section~\ref{Sec:Stanley} restricts to a bijection $\mathcal U_{2k+1}(312,231)\to\Int(\mathcal A_k)$). In Section \ref{Sec:Conclusion}, we quickly complete the enumeration of sets of the form $\mathcal U_{2k+1}(\tau^{(1)},\ldots,\tau^{(r)})$ when $\tau^{(1)},\ldots,\tau^{(r)}\in S_3$. We also formulate eighteen enumerative conjectures about sets of the form $\mathcal U_{2k+1}(\tau^{(1)},\tau^{(2)})$ with $\tau^{(1)}\in S_3$ and $\tau^{(2)}\in S_4$.

\section{Tamari Lattices and Pallo Comb Posets}\label{Sec:LatticeBack}
In this brief section, we define the Tamari 
lattices and Pallo comb posets. We will not actually need the definition of the Pallo comb posets in the rest of the article, but we include it here for the sake of completeness. 

\begin{definition}\label{Def3}
Given $\Lambda\in{\bf D}_k$, we can write $\Lambda=UD^{\gamma_1}UD^{\gamma_2}\cdots UD^{\gamma_k}$ for some nonnegative integers $\gamma_1,\ldots,\gamma_k$. Let $\lon_j(\Lambda)$ be the smallest nonnegative integer $t$ such that \[\gamma_j+\gamma_{j+1}+\cdots+\gamma_{j+t}>t.\] We call $\lon_j(\Lambda)$ the \emph{longevity} of the $j^\text{th}$ up step of $\Lambda$. The \emph{longevity sequence} of $\Lambda$ is the tuple $(\lon_1(\Lambda),\ldots,\lon_k(\Lambda))$. 
\end{definition} 
Geometrically, $\lon_j(\Lambda)$ is the semilength of the longest Dyck path that we can obtain by starting where the $j^\text{th}$ up step of $\Lambda$ ends and following $\Lambda$. For instance, the longevity sequence of the Dyck path in Figure \ref{Fig1} is $(3,0,1,0,0)$. Adding $1$ to each entry in the longevity sequence of a Dyck path produces the ``distance function" of the Dyck path, as defined in \cite{BousquetTamari}. Proposition 5 in that paper tells us that the following definition of the Tamari lattices is equivalent to the classical definitions. 

\begin{definition}\label{Def2}
Given $\Lambda,\Lambda'\in{\bf D}_k$, we write $\Lambda\leq_T\Lambda'$ if $\lon_j(\Lambda)\leq\lon_j(\Lambda')$ for all $j\in[k]$. The $k^\text{th}$ \emph{Tamari lattice} is the poset $\mathcal L_k^T=({\bf D}_k,\leq_T)$. 
\end{definition}

Theorem 2 in \cite{Pallo2} and Theorem 1 in \cite{Pallo} characterize the Tamari lattices and Pallo comb posets (defined on sets of binary trees) in terms of ``weight sequences" of binary trees. There is a bijection (described in \cite{Bernardi}) between Dyck paths and binary trees such that the weight sequence of the tree corresponding to $\Lambda\in{\bf D}_k$ is the distance function of $\Lambda$. We have used this correspondence to arrive at the next definition (we omit a justification that this definition is equivalent to others because we will not need it). 

\begin{definition}\label{Def4}
Given $\Lambda,\Lambda'\in{\bf D}_k$, we write $\Lambda\leq_{\text{Pallo}}\Lambda'$ if $\Lambda\leq_T\Lambda'$ and if for every $j\in[k]$ such that $\lon_j(\Lambda)<\lon_j(\Lambda')$, we have $\lon_{\ell}(\Lambda)\leq j-\ell-1$ for all $\ell\in[j-1]$. The $k^\text{th}$ \emph{Pallo comb poset} is $\PC_k=({\bf D}_k,\leq_{\text{Pallo}})$. 
\end{definition}

\section{Stack-Sorting Background}\label{Sec:StackBack} 
Throughout this article, permutations are finite words over the alphabet of positive integers without repeated letters (such as $47219$). Recall that $S_n$ is the set of permutations of $[n]$ and that $\mathcal U_n$ is the set of uniquely sorted permutations in $S_n$ (see Definition \ref{Def1}). The \emph{normalization} of a permutation $\pi=\pi_1\cdots\pi_n$ is the permutation in $S_n$ obtained from $\pi$ by replacing the $i^\text{th}$-smallest entry of $\pi$ by $i$ for all $i\in [n]$. We say two permutations have the \emph{same relative order} if their normalizations are equal. A permutation is called \emph{normalized} if it is in $S_n$ for some $n$. If $\sigma=\sigma_1\cdots\sigma_n$ and $\tau=\tau_1\cdots\tau_m$ are permutations, then we say $\sigma$ \emph{contains} the pattern $\tau$ if there are indices $i_1<\cdots<i_m$ such that $\sigma_{i_1}\cdots\sigma_{i_m}$ has the same relative order as $\tau$. Otherwise, we say $\sigma$ \emph{avoids} $\tau$. Let $\Av(\tau^{(1)},\tau^{(2)},\ldots)$ be the set of normalized permutations that avoid the patterns $\tau^{(1)},\tau^{(2)},\ldots$ (this sequence of patterns could be finite or infinite). Let $\Av_n(\tau^{(1)},\tau^{(2)},\ldots)=\Av(\tau^{(1)},\tau^{(2)},\ldots)\cap S_n$ and $\mathcal U_n(\tau^{(1)},\tau^{(2)},\ldots)=\Av(\tau^{(1)},\tau^{(2)},\ldots)\cap\mathcal U_n$. Let $\mathcal U(\tau^{(1)},\tau^{(2)},\ldots)$ denote the set of all uniquely sorted permutations in $\Av(\tau^{(1)},\tau^{(2)},\ldots)$. 

The investigation of permutation patterns initiated with Knuth's introduction of a certain ``stack-sorting algorithm" in \cite{Knuth}. West introduced the stack-sorting map $s$, which is a deterministic variant of Knuth's algorithm, in his dissertation \cite{West}. It follows from Knuth's analysis that $s^{-1}(123\cdots n)=\Av_n(231)$ and that $|\Av_n(231)|=C_n$. 

The fertility of a permutation $\pi$ is $|s^{-1}(\pi)|$. Bousquet-M\'elou \cite{Bousquet} provided an algorithm for determining whether or not a given permutation is sorted (i.e., has positive fertility). She then asked for a general method for computing the fertility of an arbitrary permutation. The present author accomplished this in even greater generality in \cite{DefantPostorder, DefantPreimages} by introducing new combinatorial objects called ``valid hook configurations." He also translated Bousquet-M\'elou's algorithm into the language of valid hook configurations, defining the ``canonical hook configuration" of a permutation \cite{DefantPreimages}. We are fortunate in this article that we do not need all of the definitions and main theorems concerning valid hook configurations. In order to work with uniquely sorted permutations, we will only need to define canonical hook configurations. 

The \emph{plot} of a permutation $\pi=\pi_1\cdots\pi_n$ is obtained by plotting the points $(i,\pi_i)$ for all $i\in[n]$. A \emph{hook} $H$ of $\pi$ is drawn by starting at a point $(i,\pi_i)$ in the plot of $\pi$, moving vertically upward, and then moving to the right to connect with a different point $(j,\pi_j)$. In order to do this, we must have $i<j$ and $\pi_i<\pi_j$. The point $(i,\pi_i)$ is called the \emph{southwest endpoint} of $H$, while $(j,\pi_j)$ is called the \emph{northeast endpoint} of $H$. We say a point $(r,\pi_r)$ \emph{lies strictly below} $H$ if $i<r<j$ and $\pi_r<\pi_j$. We say $(r,\pi_r)$ \emph{lies weakly below} $H$ if it lies strictly below $H$ or if $r=j$. The left image in Figure \ref{Fig3} shows the plot of a permutation along with a single hook. 

\begin{figure}[h]
\begin{center}
\includegraphics[width=.8\linewidth]{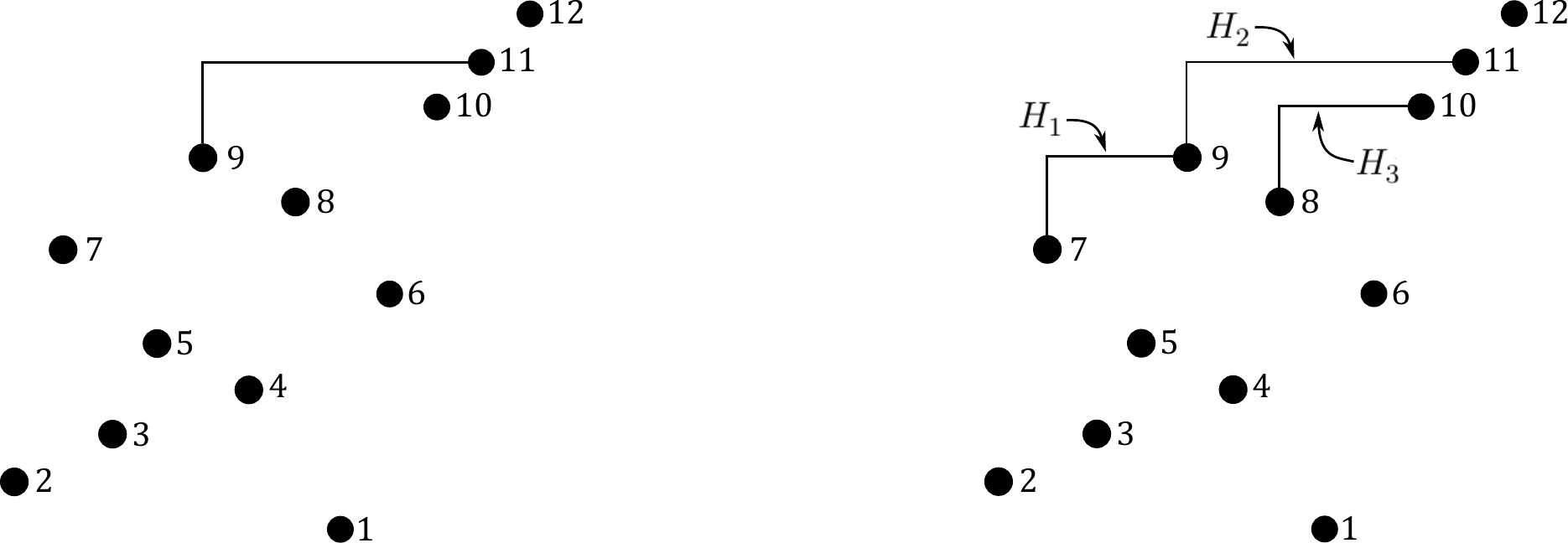}
\caption{On the left is the plot of $2\,7\,3\,5\,9\,4\,8\,1\,6\,10\,11\,12$ along with one hook whose southwest endpoint is $(5,9)$ and whose northeast endpoint is $(11,11)$. The points lying strictly below this hook are $(6,4), (7,8), (8,1), (9,6),(10,10)$. These five points and $(11,11)$ are the points lying weakly below the hook. The right image shows the canonical hook configuration of $2\,7\,3\,5\,9\,4\,8\,1\,6\,10\,11\,12$.}
\label{Fig3}
\end{center}  
\end{figure}

\noindent {\bf Canonical Hook Configuration Construction}

Recall that a descent of $\pi=\pi_1\cdots\pi_n$ is an index $i\in[n-1]$ such that $\pi_i>\pi_{i+1}$. Let $d_1<\cdots<d_k$ be the descents of $\pi$. The \emph{canonical hook configuration} of $\pi$ is the tuple $\mathcal H=(H_1,\ldots,H_k)$ of hooks of $\pi$ defined as follows. First, the southwest endpoint of the hook $H_i$ is $(d_i,\pi_{d_i})$. We let $\mathfrak N_i$ denote the northeast endpoint of $H_i$. We determine these northeast endpoints in the order $\mathfrak N_k,\mathfrak N_{k-1},\ldots,\mathfrak N_1$. First, $\mathfrak N_k$ is the leftmost point lying above and to the right of $(d_k,\pi_{d_k})$. Next, $\mathfrak N_{k-1}$ is the leftmost point lying above and to the right of $(d_{k-1},\pi_{d_{k-1}})$ that does not lie weakly below $H_k$. In general, $\mathfrak N_\ell$ is the leftmost point lying above and to the right of $(d_\ell,\pi_{d_\ell})$ that does not lie weakly below any of the hooks $H_k,H_{k-1},\ldots,H_{\ell+1}$. If there is any time during this process when the point $\mathfrak N_\ell$ does not exist, then $\pi$ does not have a canonical hook configuration. See the right part of Figure \ref{Fig3} for an example of this construction. \hspace*{\fill}$\lozenge$ 

\begin{remark}\label{Rem4}
Suppose $\pi$ is a permutation that has a canonical hook configuration $\mathcal H$. It will be useful to keep the following crucial facts in mind. 

\begin{itemize}
\item The descent tops of the plot of $\pi$ are precisely the southwest endpoints of the hooks in $\mathcal H$. 
\item No hook in $\mathcal H$ passes below a point in the plot of $\pi$. 
\item None of the hooks in $\mathcal H$ cross or overlap each other, except when the southwest endpoint of one hook coincides with the northeast endpoint of another hook. 
\end{itemize}
These observations follow immediately from the preceding construction. \hspace*{\fill}$\lozenge$
\end{remark}

The following useful proposition is a consequence of the discussion of canonical hook configurations in \cite{DefantPreimages}, although it is essentially equivalent to Bousquet-M\'elou's algorithm in \cite{Bousquet}. In combination with Theorem \ref{Thm1}, this proposition allows us to determine whether or not a given permutation is uniquely sorted. 

\begin{proposition}[\hspace{-.015cm}\cite{DefantPreimages}]\label{Prop1}
A permutation is sorted if and only if it has a canonical hook configuration. 
\end{proposition}

We end this section by recording some lemmas regarding canonical hook configurations that will prove useful in subsequent sections. Let us say a point $(i,\pi_i)$ in the plot of a permutation $\pi=\pi_1\cdots\pi_n$ is a \emph{descent top} of the plot of $\pi$ if $i$ is a descent of $\pi$. Similarly, say $(i,\pi_i)$ is a \emph{descent bottom} of the plot of $\pi$ if $i-1$ is a descent of $\pi$. The point $(i,\pi_i)$ is called a \emph{left-to-right maximum} of the plot of $\pi$ if it is higher than all of the points to its left. 

\begin{lemma}\label{Lem2}
Let $\pi$ be a uniquely sorted permutation of length $2k+1$. Let $\mathfrak N_1,\ldots,\mathfrak N_k$ be the northeast endpoints of the hooks in the canonical hook configuration of $\pi$.  Let $\DB(\pi)$ be the set of descent bottoms of the plot of $\pi$. The two $k$-element sets $\DB(\pi)$ and $\{\mathfrak N_1,\ldots,\mathfrak N_k\}$ form a partition of the set $\{(i,\pi_i):2\leq i\leq 2k+1\}$.
\end{lemma}

\begin{proof}
Theorem \ref{Thm1} tells us that we do indeed have $|\DB(\pi)|=\des(\pi)=k$. Let $d_1<\cdots<d_k$ be the descents of $\pi$, and let $\mathcal H=(H_1,\ldots,H_k)$ be the canonical hook configuration of $\pi$. We must show that $\DB(\pi)$ and $\{\mathfrak N_1,\ldots,\mathfrak N_k\}$ are disjoint. If this were not the case, then we would have $(d_\ell+1,\pi_{d_\ell+1})=\mathfrak N_m$ for some $\ell,m\in\{1,\ldots,k\}$. The southwest endpoint of $H_m$ is $(d_m,\pi_{d_m})$, and this must lie below and to the left of $\mathfrak N_m$. Thus, $m<\ell$. Also, $\mathfrak N_m$ lies strictly below the hook $H_\ell$. Referring to the canonical hook configuration construction to see how $\mathfrak N_m$ was defined, we find that this is impossible. 
\end{proof}

\begin{lemma}\label{Lem4}
Let $\pi\in\mathcal U_{2k+1}(312)$, and let $\mathcal H=(H_1,\ldots,H_k)$ be the canonical hook configuration of $\pi$. Let $\mathfrak N_i$ denote the northeast endpoint of $H_i$. The left-to-right maxima of the plot of $\pi$ are precisely the points $(1,\pi_1),\mathfrak N_1,\ldots,\mathfrak N_k$.  
\end{lemma}

\begin{proof}
Let $\DB(\pi)$ be the set of descent bottoms of the plot of $\pi$. Because $\pi$ avoids $312$, every point in the plot of $\pi$ that is not in $\DB(\pi)$ must be a left-to-right maximum of the plot of $\pi$. On the other hand, none of the points in $\DB(\pi)$ are left-to-right maxima of the plot of $\pi$. The desired result now follows from Lemma \ref{Lem2}.   
\end{proof}

\section{Permutation Operations}\label{Sec:Operators}

In this section, we establish several definitions and conventions regarding permutations. It will often be convenient to associate permutations with their plots. From this viewpoint, a permutation is essentially just an arrangements of points in the plane such that no two distinct points lie on a single vertical or horizontal line. When viewing permutations in this way, we do not distinguish between two permutations that have the same relative order. In other words, the plots that we draw are really meant to represent equivalence classes of permutations, where two permutations are equivalent if they have the same relative order. For example, if $\lambda=21$ and $\mu=12$, then $\begin{array}{l}
\includegraphics[height=1.2cm]{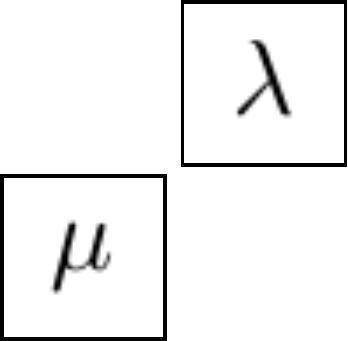}
\end{array}$ represents (the equivalence class of) $1243$. 

Given a permutation $\pi=\pi_1\cdots\pi_n\in S_n$, we let $\rev(\pi)=\pi_n\cdots\pi_1$ be the reverse of $\pi$. Let $\pi^{-1}$ be the inverse of $\pi$ in the group $S_n$; this is the permutation in $S_n$ in which the entry $i$ appears in the $\pi_i^\text{th}$ position. Geometrically, we obtain the plot of $\rev(\pi)$ by reflecting the plot of $\pi$ through the line $x=(n+1)/2$. We obtain the plot of $\pi^{-1}$ by reflecting the plot of $\pi$ though the line $y=x$. Let $\rot(\pi)$ (respectively, $\rot^{-1}(\pi)$) be the permutation whose plot is obtained by rotating the plot of $\pi$ counterclockwise (respectively, clockwise) by $90^\circ$. Equivalently, $\rot(\pi)=\rev(\pi^{-1})$. 

The \emph{sum} of two permutations $\mu$ and $\lambda$, denoted $\mu\oplus\lambda$, is the permutation obtained by placing the plot of $\lambda$ above and to the right of the plot of $\mu$. The \emph{skew sum} of $\mu$ and $\lambda$, denoted $\mu\ominus\lambda$, is the permutation obtained by placing the plot of $\lambda$ below and to the right of the plot of $\mu$. In our geometric point of view, we have \[\mu\oplus\lambda=\begin{array}{l}
\includegraphics[height=1.2cm]{CatUniquePIC4}
\end{array}\quad\text{and}\quad\mu\ominus\lambda=\begin{array}{l}
\includegraphics[height=1.2cm]{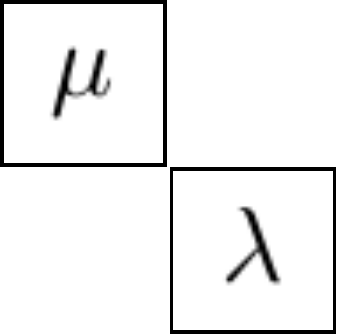}
\end{array}.\] 

For each $i\in[n]$, we define four ``sliding operators" on $S_n$. The first, denoted\footnote{The name of the operator stands for ``southwest up."} $\swu_i$, essentially takes the points in the plot of a permutation $\pi$ that lie southwest of the point with height $i$ and slides them up above all the points that are southeast of the point with height $i$. We illustrate this operator in Figure \ref{Fig4}. To define this more precisely, let $L_i$ (respectively, $R_i$) be the set of elements of $[i-1]$ that lie to the left (respectively, right) of $i$ in $\pi$. If $\pi_j\geq i$, then the $j^\text{th}$ entry of $\swu_i(\pi)$ is $\pi_j$. If $\pi_j<\pi_i$, then either $\pi_j\in L_i$ or $\pi_j\in R_i$. If $\pi_j$ is the $m^\text{th}$-smallest element of $R_i$, then the $j^\text{th}$ entry of $\swu_i(\pi)$ is $m$. If $\pi_j$ is the $m^\text{th}$-largest element of $L_i$, then the $j^\text{th}$ entry of $\swu_i(\pi)$ is $i-m$.  

\begin{figure}[h]
\begin{center}
\includegraphics[width=.55\linewidth]{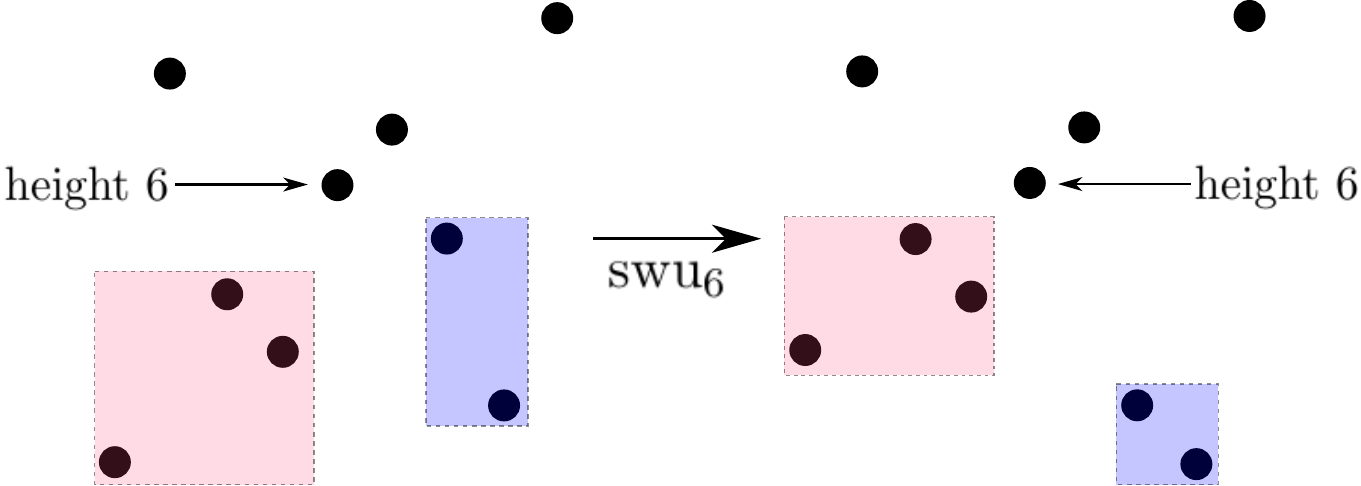}
\caption{The operator $\swu_6$ slides the points to the southwest of the point with height $6$ (shaded in pink) up. It also slides the points to the southeast of the point with height $6$ (shaded in blue) down.}
\label{Fig4}
\end{center}  
\end{figure}

The second operator we define is $\swd_i$, which takes the points in the plot of $\pi$ that lie southwest of the point with height $i$ and slides them down below the points lying to the southeast of that point. We can define this operator formally by \[\swd_i(\pi)=\rev(\swu_i(\rev(\pi))).\] The third and fourth operators, $\swl_i$ and $\swr_i$, are defined by \[\swl_i(\pi)=\rot^{-1}(\swu_i(\rot(\pi)))\quad\text{and}\quad\swr_i(\pi)=\rot^{-1}(\swd_i(\rot(\pi))).\] The operator $\swl_i$ takes the points to the southwest of the point in position $i$ and slides them to the left; the operator $\swr_i$ slides them to the right. We illustrate $\swr_6$ in Figure \ref{Fig5}. We can also define these maps on arbitrary permutations by normalizing the permutations, applying the maps, and then ``unnormalizing." For example, since $\swu_4(1243)=2341$, we have $\swu_4(2496)=4692$. 

\begin{figure}[h]
\begin{center}
\includegraphics[width=.47\linewidth]{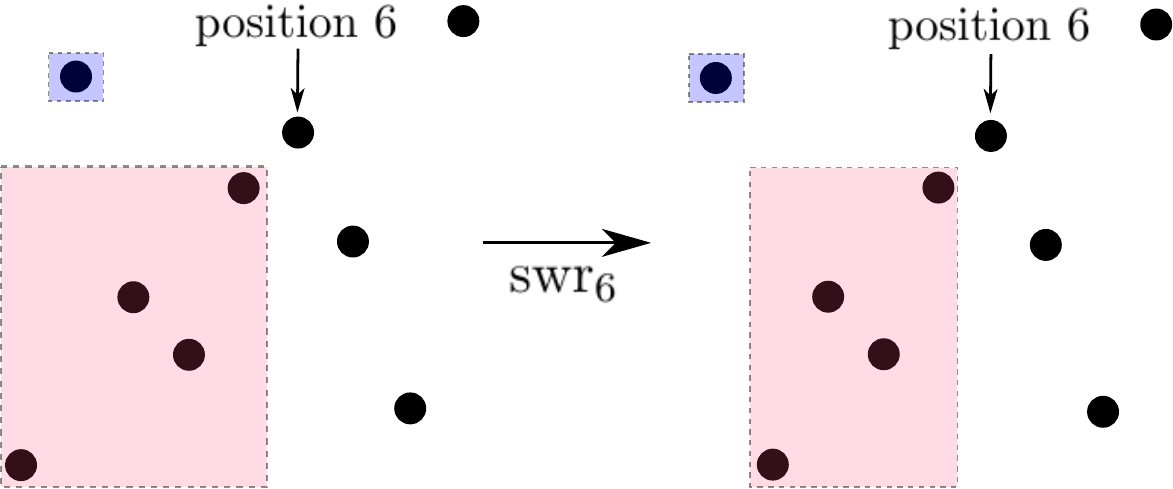}
\caption{The operator $\swr_6$ slides the points to the southwest of the point in position $6$ (shaded in pink) to the right. It moves the point to the northeast of the point in position $6$ (shaded in blue) to the left. }
\label{Fig5}
\end{center}  
\end{figure}

\begin{remark}\label{Rem3}
One simple consequence of the above definition is that after we apply $\swu_i$ (respectively, $\swr_i$) to a permutation, the entry with height $i$ (respectively, position $i$) cannot form the highest entry (respectively, rightmost entry) in a $132$ pattern. A similar remark applies to $\swd_i$ and $\swl_i$ as well. \hspace*{\fill}$\lozenge$
\end{remark}

We define $\swu:S_n\to S_n$ by \[\swu=\swu_1\circ\swu_2\circ\cdots\circ\swu_n.\] An alternative recursive way of thinking of this map, which we illustrate in Figure \ref{Fig6}, is as follows. Let us write $\pi=LnR$. We have \[\swu(\pi)=(\swu(L)\oplus 1)\ominus\swu(R).\] This recursive definition requires us to define $\swu$ on arbitrary permutations, which we can do by normalizing, applying $\swu$, and then unnormalizing. The reader should imagine this sliding operator as acting on a collection of points in the plane instead of a string of numbers. To see that these two definitions of $\swu$ coincide, let $\ell$ be the length of the subpermutation $L$ of $\pi$. The map $\swu_n$ sends $\pi$ to $(L\oplus 1)\ominus R$. When we apply the map $\swu_{n-\ell}\circ\cdots\circ\swu_{n-1}$ to $(L\oplus 1)\ominus R$, it does not move any of the points in the plot of $R$. It does, however, change $L$ into $\swu(L)$ so that $\swu_{n-\ell}\circ\cdots\circ\swu_n(\pi)=(\swu(L)\oplus 1)\ominus R$. Similarly, the map $\swu_1\circ\cdots\circ\swu_{n-\ell-1}$ only has the effect of changing $R$ into $\swu(R)$ so that $\swu(\pi)=\swu_1\circ\cdots\circ\swu_n(\pi)=(\swu(L)\oplus 1)\ominus\swu(R)$. 

\begin{figure}[h]
\begin{center}
\includegraphics[width=.55\linewidth]{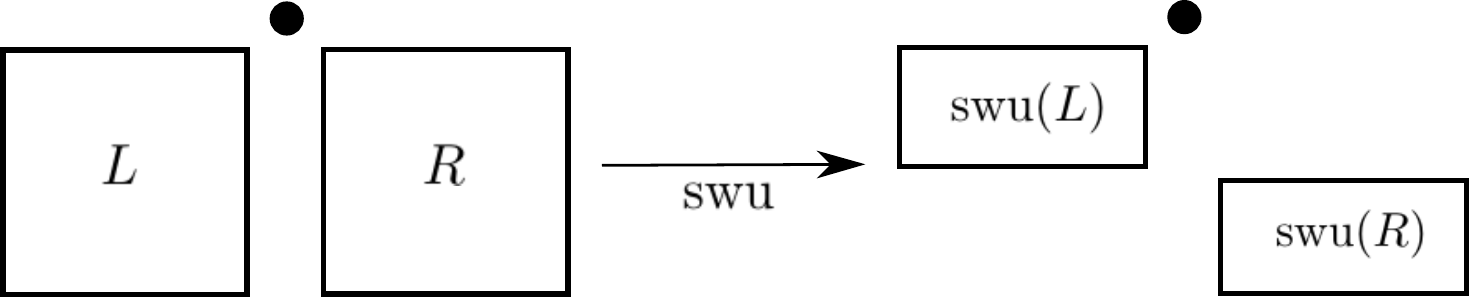}
\caption{The recursive definition of the map $\swu$.}
\label{Fig6}
\end{center} 
\end{figure} 

Similarly, let \[\swd=\swd_1\circ\swd_2\circ\cdots\circ\swd_n,\quad\swl=\swl_1\circ\swl_2\circ\cdots\circ\swl_n,\quad \swr=\swr_1\circ\swr_2\circ\cdots\circ\swr_n.\] As before, we can also define these maps for arbitrary permutations. By a slight abuse of notation, we use the symbols $\swu,\swd,\swl,\swr$ to denote the maps defined on all permutations of all lengths (or alternatively, on their equivalence classes). 

\begin{lemma}\label{Lem5}
The maps \[\swu:\Av(231)\to\Av(132)\quad\text{and}\quad\swd:\Av(132)\to\Av(231)\] are bijections that are inverses of each other. The maps \[\swl:\Av(132)\to\Av(312)\quad\text{and}\quad\swr:\Av(312)\to\Av(132)\] are bijections that are inverses of each other. 
\end{lemma}
\begin{proof}
We first prove by induction on $n$ that \[\swu:\Av_n(231)\to\Av_n(132)\quad\text{and}\quad\swd:\Av_n(132)\to\Av_n(231)\] are bijections that are inverses of each other. This is clear if $n\leq 2$, so assume $n\geq 3$. Choose $\pi\in\Av_n(231)$, and write $\pi=LnR$. Because $\pi$ avoids $231$, we have $\pi=L\oplus(1\ominus R)$. Furthermore, $L$ and $R$ avoid $231$. The recursive definition of $\swu$ tells us that $\swu(\pi)=(\swu(L)\oplus 1)\ominus\swu(R)$. By induction, we find that $\swu(L)$ and $\swu(R)$ avoid $132$, so $\swu(\pi)$ also avoids $132$. Moreover, there is a recursive definition of $\swd$ analogous to the recursive definition of $\swu$ that yields \[\swd(\swu(\pi))=\swd((\swu(L)\oplus 1)\ominus\swu(R))=\swd(\swu(L))\oplus(1\ominus\swd(\swu(R))).\] By induction on $n$, this is just $L\oplus(1\ominus R)$, which is $\pi$. This shows that $\swd$ is a left inverse of $\swu$, and a similar argument with the roles of $\swd$ and $\swu$ reversed shows that $\swd$ is also a right inverse of $\swu$. The second statement now follows easily from the first if we use the fact that $\swl=\rot^{-1}\circ\swu\circ\rot$.
\end{proof}

\begin{lemma}\label{Lem10}
The maps \[\swu:\Av(231,312)\to\Av(132,312)\quad\text{and}\quad\swd:\Av(132,312)\to\Av(231,312)\] are bijections that are inverses of each other. The maps \[\swl:\Av(132,231)\to\Av(231,312)\quad\text{and}\quad\swr:\Av(231,312)\to\Av(132,231)\] are bijections that are inverses of each other. 
\end{lemma}

\begin{proof}
To prove the first statement, we must show that $\swu(\Av_n(231,312))=\Av_n(132,312)$ for all $n\geq 1$. This is clear if $n\leq 2$, so we may assume $n\geq 3$ and induct on $n$. For any given $\pi\in\Av_n(231,312)$, we can use the fact that $\pi$ avoids $231$ to write $\pi=L\oplus (1\ominus R)$ for some permutations $L,R\in \Av(231,312)$. Note that $R$ is a decreasing permutation because $\pi$ avoids $312$. By induction, $\swu(L)\in\Av(132,312)$. Also, $\swu(R)=R$. The recursive definition of $\swu$ tells us that $\swu(\pi)=(\swu(L)\oplus 1)\ominus\swu(R)$. The permutation $\swu(\pi)$ certainly avoids $132$ by Lemma~\ref{Lem5}. Since $\swu(L)$ avoids $312$ and $\swu(R)=R$ is decreasing, $\swu(\pi)$ avoids $312$. This proves that $\swu(\Av_n(231,312))\subseteq\Av_n(132,312)$ for all $n\geq 1$. A similar argument with $\swu$ replaced by $\swd$ proves the reverse containment. The second statement now follows from the first if we use the fact that $\swl=\rot^{-1}\circ\rev\circ\swd\circ\rev\circ\rot$.   
\end{proof}

In the following lemma, recall that $\Des(\pi)$ denotes the set of descents of $\pi$.

\begin{lemma}\label{Lem1}
For every permutation $\pi$, we have $\Des(\swu(\pi))=\Des(\swd(\pi))=\Des(\pi)$. If $\pi\in\Av(312)$, then $\des(\swr(\pi))=\des(\pi^{-1})=\des(\pi)$. If $\pi\in\Av(132)$, then $\des(\swl(\pi))=\des(\pi^{-1})=\des(\pi)$. 
\end{lemma}

\begin{proof}
For each $i\in[n]$ and $\sigma\in S_n$, it is clear from the definitions of $\swu_i$ and $\swd_i$ that $\Des(\swu_i(\sigma))=\Des(\swd_i(\sigma))=\Des(\sigma)$. The first claim now follows from the definitions $\swu(\pi)=\swu_1\circ\cdots\circ\swu_n(\pi)$ and $\swd(\pi)=\swd_1\circ\cdots\circ\swd_n(\pi)$. Now assume $\pi\in\Av_n(312)$. The second claim is trivial if $n\leq 1$, so we may assume $n\geq 2$ and induct on $n$. Because $\pi$ avoids $312$, we can write $\pi=\lambda\oplus(\mu\ominus 1)$ for some $\lambda,\mu\in \Av(312)$. We have $\pi^{-1}=\lambda^{-1}\oplus(1\ominus\mu^{-1})$, so we can use induction to see that $\des(\pi)=\des(\lambda)+\des(\mu)+\delta=\des(\lambda^{-1})+\des(\mu^{-1})+\delta=\des(\pi^{-1})$, where $\delta=1$ if $\mu$ is nonempty and $\delta=0$ if $\mu$ is empty. Similarly, the recursive definition of $\swr$ (which is analogous to the recursive definition that we gave for $\swu$) tells us that $\swr(\pi)=\swr(\mu)\ominus(\swr(\lambda)\oplus 1)$. We know by induction that $\des(\swr(\mu))=\des(\mu)$ and $\des(\swr(\lambda))=\des(\lambda)$. Consequently, \[\des(\swr(\pi))=\des(\swr(\mu))+\des(\swr(\lambda))+\delta=\des(\mu)+\des(\lambda)+\delta=\des(\pi).\] The proof of the third claim is completely analogous to the proof of the second.  
\end{proof}

\begin{lemma}\label{Lem6}
If $\pi$ is a sorted permutation, then $\swu(\pi)$ and $\swd(\pi)$ are sorted. If, in addition, $\pi$ avoids $132$, then $\swl(\pi)$ is sorted. 
\end{lemma}

\begin{proof}
Assume $\sigma\in S_n$ is sorted, and let $\mathcal H$ be its canonical hook configuration, which is guaranteed to exist by Proposition \ref{Prop1}. For $i\in[n]$, we claim that $\swu_i(\sigma)$ is sorted. To see this, note that the plot of $\swu_i(\sigma)$ is obtained from the plot of $\sigma$ by sliding some points up and sliding other points down. During this process, let us simply keep the hooks in $\mathcal H$ attached to their southwest and northeast endpoints. This is illustrated in Figure \ref{Fig7}. Since $\swu_i$ does not change the relative order of the points on any one side of the point with height $i$, the resulting configuration of hooks is the canonical hook configuration of $\swu_i(\sigma)$.\footnote{It follows from the results in \cite{DefantPolyurethane} that $\sigma$ and $\swu_i(\sigma)$ actually have the same fertility, but we do not need the full strength of that result in this article.} Crucially, we are using the fact, which follows from the canonical hook configuration construction, that no hooks in $\mathcal H$ pass below the point with height $i$ in the plot of $\sigma$ (see the second bullet in Remark~\ref{Rem4}). We are also using the fact that $\Des(\swu_i(\sigma))=\Des(\sigma)$. A similar argument shows that $\swd_i(\sigma)$ is sorted. As $i$ and $\sigma$ were arbitrary, we find that if $\pi$ is sorted, then $\swu(\pi)=\swu_1\circ\cdots\circ\swu_n(\pi)$ and $\swd(\pi)=\swd_1\circ\cdots\circ\swd_n(\pi)$ are sorted. 

\begin{figure}[h]
\begin{center}
\includegraphics[width=.63\linewidth]{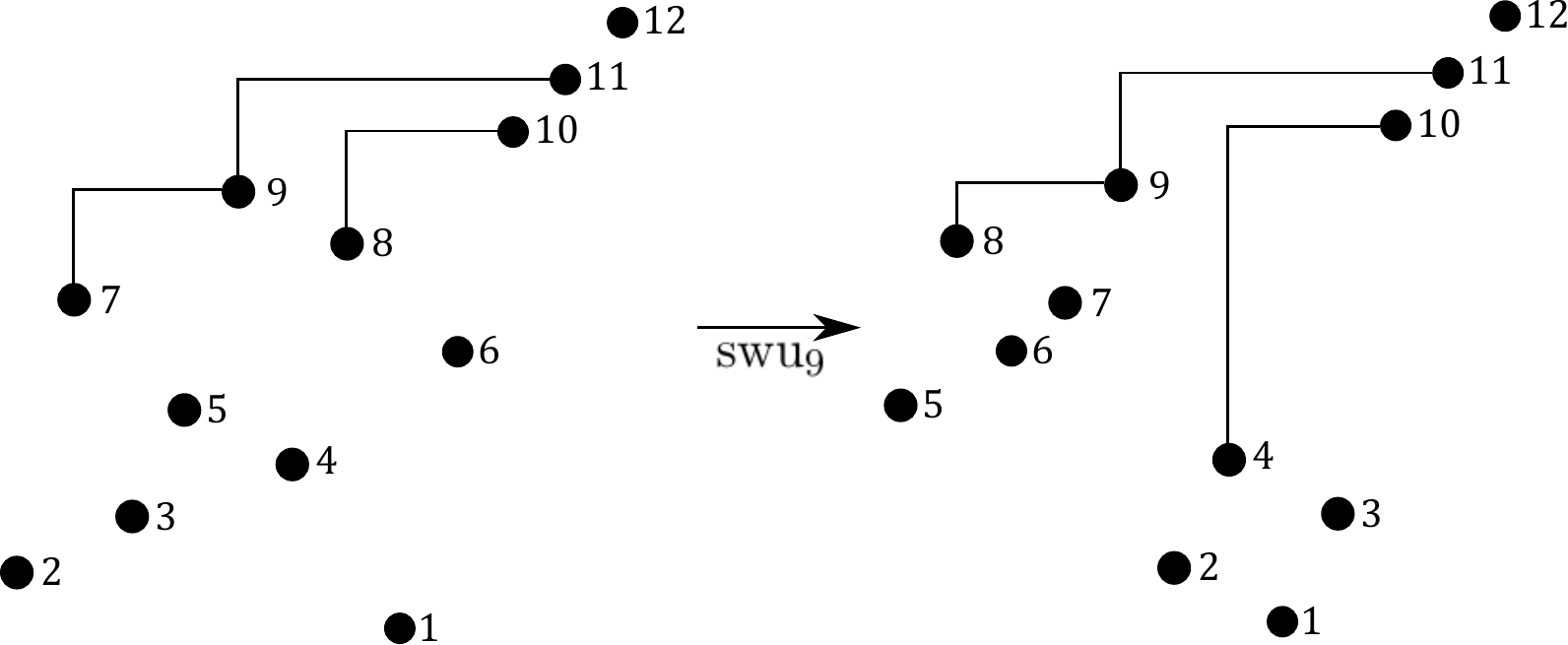}
\caption{The canonical hook configuration of $\sigma$ transforms into the canonical hook configuration of $\swu_9(\sigma)$.}
\label{Fig7}
\end{center} 
\end{figure} 

To help us prove the second statement, let us define the \emph{deficiency} of $\sigma$, denoted $\de(\sigma)$, to be the smallest nonnegative integer $\ell$ such that $\sigma\oplus(123\cdots\ell)$ is sorted. Such an $\ell$ is guaranteed to exist by Corollary 2.3 in \cite{Bousquet}. Thus, $\de(\sigma)=0$ if and only if $\sigma$ is sorted. Roughly speaking, one can think of $\de(\sigma)$ as the number of descent tops in the plot of $\sigma$ that cannot find corresponding northeast endpoints for their hooks during the canonical hook configuration construction (see Figure \ref{Fig15}). Indeed, these descent tops can use the $\ell$ points that are added to the plot of $\sigma$ as their northeast endpoints in the canonical hook configuration of $\sigma\oplus(123\cdots\ell)$. 

\begin{figure}[h]
\begin{center}
\includegraphics[width=.2\linewidth]{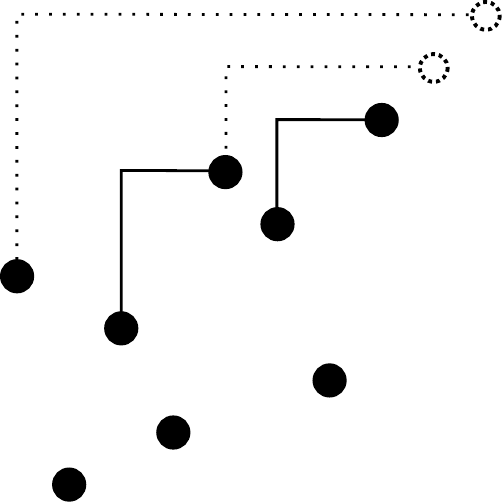}
\caption{The points represented by black discs form the plot of the permutation $51427638$. This permutation is not sorted because the descent tops $(1,5)$ and $(5,7)$ cannot find northeast endpoints for their hooks. If we were to add the additional two dotted circles as points, we would obtain a sorted permutation because we could add in the dotted hooks. Therefore, $\de(\pi)=2$.}
\label{Fig15}
\end{center} 
\end{figure} 

From this interpretation of deficiency, one can verify that for every permutation $\sigma$ and every nonempty permutation $\tau$, we have 
\begin{equation}\label{Eq10}
\de(\sigma\oplus\tau)\leq\de(\sigma\oplus 1)+\de(\tau).
\end{equation} If $\sigma$ is also nonempty, then
\begin{equation}\label{Eq17}
\de(\tau\ominus\sigma)=\de(\tau)+\de(\sigma)+1.
\end{equation} 
We claim that if $\pi$ avoids $132$, then $\de(\swl(\pi))\leq\de(\pi)$. In particular, this will prove that if $\pi$ avoids $132$ and is sorted, then $\swl(\pi)$ is sorted. The proof of this claim is by induction on the length $n$ of the permutation $\pi$. We are done if $n\leq 2$ since $\swl(\pi)=\pi$ in that case, so we may assume that $n\geq 3$. Since $\pi$ avoids $132$, we can write $\pi=\mu\ominus(\lambda\oplus 1)$ for some permutations $\lambda$ and $\mu$ that avoid $132$. There is a recursive definition of $\swl$, which is completely analogous to (and also follows from) the recursive definition of $\swu$, that tells us that $\swl(\pi)=\swl(\lambda)\oplus(\swl(\mu)\ominus 1)$. By induction, 
\begin{equation}\label{Eq12}
\de(\swl(\lambda))\leq\de(\lambda)\quad\text{and}\quad\de(\swl(\mu))\leq\de(\mu).
\end{equation} It follows immediately from the definition of deficiency and the first inequality in \eqref{Eq12} that 
\begin{equation}\label{Eq11}
\de(\swl(\lambda)\oplus 1)\leq\de(\lambda\oplus 1).
\end{equation}
If $\mu$ is nonempty, then we can apply \eqref{Eq10}, \eqref{Eq17}, \eqref{Eq12}, and \eqref{Eq11} to find that \[\de(\swl(\pi))=\de(\swl(\lambda)\oplus(\swl(\mu)\ominus 1))\leq\de(\swl(\lambda)\oplus 1)+\de(\swl(\mu)\ominus 1)\] \[=\de(\swl(\lambda)\oplus 1)+\de(\swl(\mu))+\de(1)+1\leq\de(\lambda\oplus 1)+\de(\mu)+1=\de(\mu\ominus(\lambda\oplus 1))=\de(\pi).\] If $\mu$ is empty, then it follows from \eqref{Eq11} that $\de(\swl(\pi))=\de(\swl(\lambda)\oplus 1)\leq\de(\lambda\oplus 1)=\de(\pi)$.      
\end{proof}

\section{Stanley Intervals and $\mathcal U_{2k+1}(312)$}\label{Sec:Stanley}

Let us begin by establishing some notation that will be useful in this section and the next. Let $n=2k+1$, and assume $\pi=\pi_1\cdots\pi_n\in \mathcal U_n(312)$. Let $\Lambda_i=D$ if $n-i\in\Des(\pi)$, and let $\Lambda_i=U$ otherwise. Let $\Lambda_i'=U$ if $i\in\Des(\rot(\pi))$, and let $\Lambda_i'=D$ otherwise. Form the words $\Lambda=\Lambda_1\cdots\Lambda_{2k}$ and $\Lambda'=\Lambda_1'\cdots\Lambda_{2k}'$ over the alphabet $\{U,D\}$, and let $\DL_k(\pi)=(\Lambda,\Lambda')$. Figure \ref{Fig8} illustrates this procedure. Our goal is to show that $\DL_k(\pi)\in\Int(\mathcal L_k^S)$ and that the resulting map \[\DL_k:\mathcal U_{2k+1}(312)\to\Int(\mathcal L_k^S)\] is bijective.\footnote{We use the letter $\Lambda$ to denote Dyck paths because it resembles an up step followed by a down step. The ``double lambda" symbol $\DL$ is meant to resemble two copies of $\Lambda$ since the bijections output pairs of Dyck paths.}

\begin{figure}[h]
\begin{center}
\includegraphics[width=.4\linewidth]{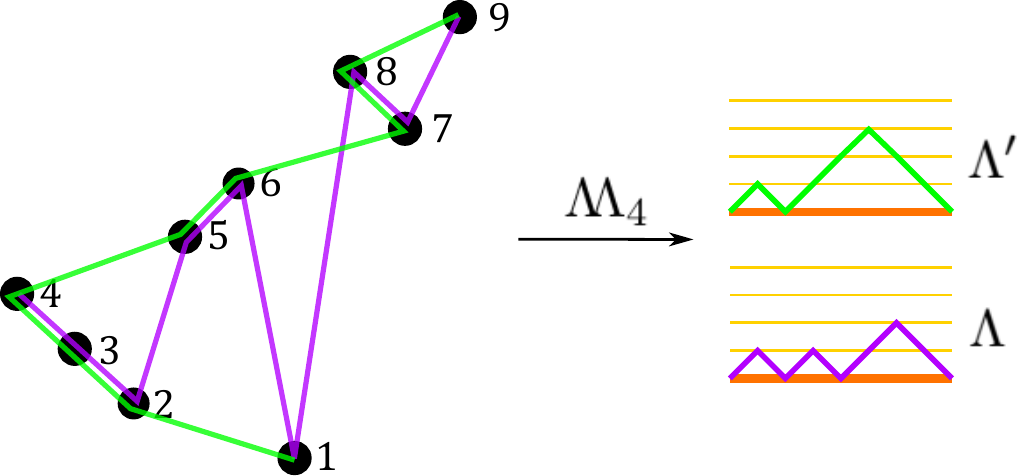}
\caption{An example illustrating the definition of $\DL_4$. Draw a purple path on the plot of the permutation by moving from right to left and connecting the dots along the way. Similarly, draw a green path by moving from top to bottom and connecting the dots along the way. Now rotate the purple path $180^{\circ}$ to obtain the purple Dyck path on the bottom. Similarly, rotate the green path by $90^{\circ}$ clockwise to obtain the \emph{reverse} of the green Dyck path on the top.}
\label{Fig8}
\end{center} 
\end{figure} 

Lemma \ref{Lem4} tells us that the left-to-right maxima of the plot of $\pi$ are $(1,\pi_1),\mathfrak N_1,\ldots,\mathfrak N_k$, where $\mathfrak N_1,\ldots,\mathfrak N_k$ are the northeast endpoints of the hooks in the canonical hook configuration of $\pi$. It will be useful to keep in mind that $\pi_n=n$ because $\pi$ is sorted. Let $\mathfrak R_0,\ldots,\mathfrak R_k$ be these left-to-right maxima, written in order from right to left (for example, $\mathfrak R_0=(n,n)$ and $\mathfrak R_k=(1,\pi_1)$). To improve readability, we also let $\mathfrak P(i)=(i,\pi_i)$. 

From $\pi$, we obtain a $k\times k$ matrix $M(\pi)=(m_{ij})$ by letting $m_{ij}$ be the number of points in the plot of $\pi$ that lie between $\mathfrak R_{k-j}$ and $\mathfrak R_{k-j+1}$ horizontally and lie between $\mathfrak R_i$ and $\mathfrak R_{i+1}$ vertically (we make the convention that all points ``lie above $\mathfrak R_{k+1}$," even though $\mathfrak R_{k+1}$ is not actually a point that we have defined).\footnote{Throughout this article, when we say a point $\mathfrak X$ lies \emph{horizontally between} two points $\mathfrak X'$ and $\mathfrak X''$, we mean that the position (i.e., $x$-coordinate) of $\mathfrak X$ is strictly between the positions of $\mathfrak X'$ and $\mathfrak X''$. Similarly, we say $\mathfrak X$ lies \emph{vertically between} $\mathfrak X'$ and $\mathfrak X''$ if the height (i.e., $y$-coordinate) of $\mathfrak X$ is strictly between the heights of $\mathfrak X'$ and $\mathfrak X''$.} Observe that $m_{ij}=0$ whenever $j\leq k-i$ (in other words, the matrix obtained by ``flipping $M(\pi)$ upside down" is upper-triangular). Alternatively, we can imagine drawing vertical lines through the points $\mathfrak R_0,\ldots,\mathfrak R_k$ and horizontal lines through $\mathfrak R_1,\ldots,\mathfrak R_k$ to produces an array of cells as in Figure \ref{Fig9}. The matrix $M(\pi)$ is now obtained by recording the number of points in each of these cells. 

\begin{figure}[h]
\[
\begin{array}{l}\includegraphics[width=.23\linewidth]{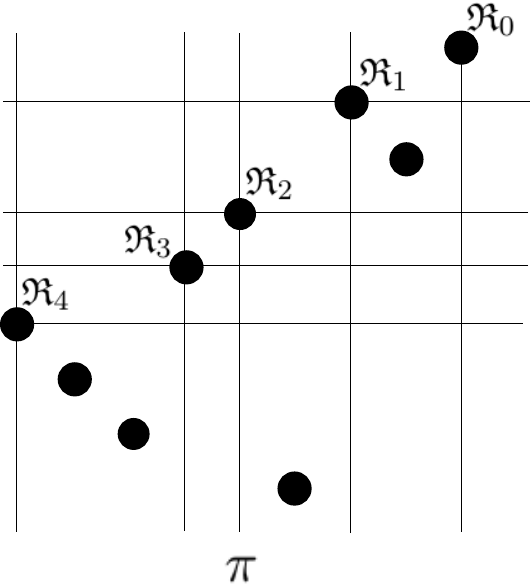}\end{array}\quad M(\pi)=\left( \begin{array}{cccc}
0 & 0 & 0 & 1 \\
0 & 0 & 0 & 0 \\
0 & 0 & 0 & 0 \\
2 & 0 & 1 & 0 \end{array} \right)\]
\caption{The array of cells and the corresponding matrix $M(\pi)$.}
\label{Fig9}
\end{figure} 

\begin{remark}\label{Rem1}
In the array of cells we have just described, the points appearing in each column must be decreasing in height from left to right because $\pi$ avoids $312$. Similarly, the points appearing in each row must be decreasing in height from left to right. This tells us that the permutation $\pi$ is uniquely determined by the matrix $M(\pi)$. Indeed, the matrix tells us how many points to place in each cell, and the positions of all of the points relative to each other are then determined by the fact that the points within rows and columns are decreasing in height. \hspace*{\fill}$\lozenge$  
\end{remark}

We now return to the definition of $\DL_k$. One can check that $\Lambda=UD^{\gamma_1}UD^{\gamma_2}\cdots UD^{\gamma_k}$ and $\Lambda'=UD^{\gamma_1'}UD^{\gamma_2'}\cdots UD^{\gamma_k'}$, where $\gamma_i$ is the sum of the entries in column $k-i+1$ of $M(\pi)$ and $\gamma_i'$ is the sum of the entries in row $i$ of $M(\pi)$. Because every nonzero entry in one of the first $i$ rows of $M(\pi)$ is also in one of the last $i$ columns of $M(\pi)$, we have 
\begin{equation}\label{Eq5}
\gamma_1+\cdots+\gamma_i\geq\gamma_1'+\cdots+\gamma_i'\quad\text{for all }i\in\{1,\ldots,k\}.
\end{equation} 

\begin{lemma}
Preserving the notation from above, we have $\DL_k(\pi)\in\Int(\mathcal L_k^S)$.
\end{lemma}
\begin{proof}
Let us check that $\Lambda$ and $\Lambda'$ are actually Dyck paths. Because $\pi$ has $k$ descents, exactly half of the letters in $\Lambda$ are equal to $U$. We can use Lemma \ref{Lem1} and the fact that $\rot(\pi)=\rev(\pi^{-1})$ to see that $\des(\rot(\pi))=\des(\rev(\pi^{-1}))=2k-\des(\pi^{-1})=2k-\des(\pi)=k$. Therefore, exactly half of the letters in $\Lambda'$ are equal to $U$. 

Recall that $n=2k+1$. Choose $p\in\{1,\ldots,2k\}$, and let $u$ be the number of appearances of the letter $U$ in $\Lambda_1\cdots\Lambda_p$. For $i\in\{1,\ldots,2k\}$, we have $\Lambda_i=D$ if and only if $\mathfrak P(n-i+1)$ is in $\DB(\pi)$, the set of descent bottoms in the plot of $\pi$. Because $\DB(\pi)$ and $\{\mathfrak N_1,\ldots,\mathfrak N_k\}$ form a partition of $\{\mathfrak P(r):2\leq r\leq n\}$ by Lemma \ref{Lem2}, we have $\Lambda_i=U$ if and only if $\mathfrak P(n-i+1)\in\{\mathfrak N_1,\ldots,\mathfrak N_k\}$. This means that $u$ is the number of northeast endpoints (in the canonical hook configuration of $\pi$) that lie in the set $\{\mathfrak P(n-p+1),\ldots,\mathfrak P(n)\}$. Also, $p-u$ is the number of appearances of $D$ in $\Lambda_1\cdots\Lambda_p$, which is $|\Des(\pi)\cap\{n-p,\ldots,n-1\}|$. This is the number of southwest endpoints that lie in the set $\{\mathfrak P(n-p),\ldots,\mathfrak P(n-1)\}$. Each of these southwest endpoints must belong to a hook whose northeast endpoint is in $\{\mathfrak P(n-p+1),\ldots,\mathfrak P(n)\}$, so $p-u\leq u$. As $p$ was arbitrary, this proves that $\Lambda$ is a Dyck path. 

We can write \[\Lambda=UD^{\gamma_1}UD^{\gamma_2}\cdots UD^{\gamma_k}\quad\text{and}\quad\Lambda'=UD^{\gamma_1'}UD^{\gamma_2'}\cdots UD^{\gamma_k'}.\] For every $i\in\{1,\ldots,k\}$, the quantity $i-\sum_{j=1}^i\gamma_j$ (respectively, $i-\sum_{j=1}^i\gamma_j'$) is the height of the lowest point where $\Lambda$ (respectively, $\Lambda'$) intersects the line $y=-x+2i$. In order to see that the path $\Lambda'$ lies weakly above $\Lambda$, it suffices to check that $i-\sum_{j=1}^i\gamma_j\leq i-\sum_{j=1}^i\gamma_j'$ for every $i\in\{1,\ldots,k\}$. This is precisely the inequality \eqref{Eq5}, so it follows that $\Lambda'$ is a Dyck path satisfying $\Lambda\leq_S\Lambda'$.  
\end{proof}

We need one additional technical lemma before we can prove the invertibility of $\DL_k$. Given a $k\times k$ matrix $M=(m_{ij})$ and indices $r,r',c,c'\in\{1,\ldots,k\}$, consider the matrix obtained by deleting all rows of $M$ except rows $r$ and $r'$ and deleting all columns of $M$ except columns $c$ and $c'$. We say this new matrix is a \emph{lower $2\times 2$ submatrix} of $M$ if $k+1-c\leq r<r'$ and $c<c'$. 

\begin{lemma}\label{Lem3}
Let $a_1,\ldots,a_k,b_1,\ldots,b_k$ be nonnegative integers such that $a_1+\cdots+a_k=b_1+\cdots+b_k$ and $a_{k-i+1}+\cdots+a_k\leq b_{k-i+1}+\cdots+b_k$ for all $i\in\{1,\ldots,k\}$. There exists a unique $k\times k$ matrix $M=(m_{ij})$ with nonnegative integer entries such that 
\begin{enumerate}[(i)]
\item $m_{ij}=0$ whenever $j\leq k-i$;
\item the sum of the entries in column $i$ of $M$ is $b_i$ for every $i\in\{1,\ldots,k\}$;
\item the sum of the entries in row $i$ of $M$ is $a_{k-i+1}$ for every $i\in\{1,\ldots,k\}$;
\item in every lower $2\times 2$ submatrix of $M$, either the bottom left entry or the top right entry is $0$.  
\end{enumerate}
\end{lemma} 

\begin{proof}
Let us first prove that there is a $k\times k$ matrix $M$ satisfying properties (i)--(iii). Let $R=a_1+\cdots+a_k=b_1+\cdots+b_k$. We induct on both $k$ and $R$, observing that the proof is trivial if $k=1$ or $R=0$. Assume $k\geq 2$ and $R\geq 1$. Let us first consider the case in which $b_k=0$. Since $a_k\leq b_k$, we have $a_k=0$ as well. Notice that $a_1+\cdots+a_{k-1}=b_1+\cdots+b_{k-1}$ and $a_{(k-1)-i+1}+\cdots+a_{k-1}\leq b_{(k-1)-i+1}+\cdots+b_{k-1}$ for all $i\in\{1,\ldots,k-1\}$. Using the induction hypothesis (inducting on $k$), we find that there is a matrix $M'=(m_{ij}')$ such that the properties (i)--(iii) are satisfied when we replace $k$ by $k-1$ and replace $M$ by $M'$. Now let $m_{ij}=0$ when $i=1$ or $j=k$, and let $m_{ij}=m_{(i-1)j}'$ when $i\geq 2$ and $j\leq k-1$. The matrix $M=(m_{ij})$ satisfies properties (i)--(iii).  

We now consider the case in which $b_k\geq 1$. Let $\ell$ be the smallest positive integer such that $a_{k-\ell+1}\geq 1$. Let $b_i'=b_i$ for $i\neq k$, and let $b_k'=b_k-1$. Let $a_i'=a_i$ for $i\neq k-\ell+1$, and let $a_{k-\ell+1}'=a_{k-\ell+1}-1$. Note that $a_1'+\cdots+a_k'=b_1'+\cdots+b_k'=R-1$. If $i<\ell$, then $a_{k-i+1}'+\cdots+a_k'=a_{k-i+1}+\cdots+a_k=0\leq b_{k-i+1}'+\cdots+b_k'$. If $i\geq\ell$, then $a_{k-i+1}'+\cdots+a_k'=a_{k-i+1}+\cdots+a_k-1\leq b_{k-i+1}+\cdots+b_k-1=b_{k-i+1}'+\cdots+b_k'$. This shows that the hypotheses of the lemma are satisfied by $a_1',\ldots,a_k',b_1',\ldots,b_k'$. By induction on $R$, we see that there is a matrix $M'=(m_{ij}')$ such that properties (i)--(iii) are satisfied when we replace $a_1,\ldots,a_k,b_1,\ldots,b_k$ by $a_1',\ldots,a_k',b_1',\ldots,b_k'$ and replace $M$ by $M'$. Let $m_{ij}=m_{ij}'$ when $(i,j)\neq(k-\ell+1,k)$, and let $m_{(k-\ell+1)k}=m_{(k-\ell+1)k}'+1$. The matrix $M=(m_{ij})$ satisfies properties (i)--(iii).

To complete the proof of existence, we define the \emph{energy} of a $k\times k$ matrix $N=(n_{ij})$ with nonnegative integer entries to be $e(N)=\sum_{i=1}^k\sum_{j=1}^k 2^{{i-j}}n_{ij}$. If $M$ does not satisfy property (iv), then we can define a ``move" on $M$ as follows. Choose $r,r',c,c'\in\{1,\ldots,k\}$ with $k+1-c\leq r<r'$ and $c<c'$ such that $m_{r'c}$ and $m_{rc'}$ are positive. Now replace the entries $m_{rc},m_{rc'},m_{r'c},m_{r'c'}$ with the entries $m_{rc}+1,m_{rc'}-1,m_{r'c}-1,m_{r'c'}+1$, respectively. Performing a move produces a new matrix $\widetilde M$ that still satisfies properties (i)--(iii). Considering the energies of these matrices, we find that \[e(M)-e(\widetilde M)=-2^{r-c}+2^{r-c'}+2^{r'-c}-2^{r'-c'}\geq-2^{(r'-1)-c}+2^{r-c'}+2^{r'-c}-2^{r'-(c+1)}=2^{r-c'}\geq 2^{1-k}.\] This shows that after applying a finite sequence of moves, we will eventually obtain a matrix that satisfies all of the properties (i)--(iv). 

To prove uniqueness, we assume by way of contradiction that there are two distinct matrices $M=(m_{ij})$ and $M'=(m_{ij}')$ satisfying the properties stated in the lemma. Because they are distinct, we can find a pair $(i_0,j_0)$ with $m_{i_0j_0}\neq m_{i_0j_0}'$. We may assume that $j_0$ was chosen maximally, which means $m_{ij}=m_{ij}'$ whenever $j>j_0$. We may assume that $i_0$ was chosen maximally after $j_0$ was chosen, meaning $m_{ij_0}=m_{ij_0}'$ whenever $i>i_0$. We may assume without loss of generality that $m_{i_0j_0}>m_{i_0j_0}'$. Because $M$ and $M'$ satisfy property (ii), their $j_0^\text{th}$ columns have the same sum. This means that there exists $i_1\neq i_0$ with $m_{i_1j_0}<m_{i_1j_0}'$. In particular, $m_{i_1j_0}'$ is positive. The maximality of $i_0$ guarantees that $i_1<i_0$. Because $M$ and $M'$ satisfy property (iii), their $i_1^\text{th}$ rows have the same sum. This means that there exists $j_1\neq j_0$ with $m_{i_1j_1}>m_{i_1j_0}'$. The maximality of $j_0$ guarantees that $j_1<j_0$. Since $M'$ satisfies property (i) and $m_{i_1j_1}>0$, we must have $k+1-j_1\leq i_1$. Now, the $j_1^\text{th}$ columns of $M$ and $M'$ have the same sum, so there exists $i_2\neq i_1$ such that $m_{i_2j_1}<m_{i_2j_1}'$. If $i_2>i_1$, then $m_{i_2j_1}'$ and $m_{i_1j_0}'$ are positive numbers that form the bottom left and top right entries in a lower $2\times 2$ submatrix of $M$. This is impossible since $M$ satisfies property (iv), so we must have $i_2<i_1$. Continuing in this fashion, we find decreasing sequences of positive integers $i_0>i_1>i_2>\cdots$ and $j_0>j_1>j_2>\cdots$. This is our desired contradiction. 
\end{proof}

\begin{theorem}\label{Thm2}
For each nonnegative integer $k$, the map $\DL_k:\mathcal U_{2k+1}(312)\to\Int(\mathcal L_k^S)$ is a bijection. 
\end{theorem}
\begin{proof}
We first prove surjectivity. Fix $(\Lambda,\Lambda')\in\Int(\mathcal L_k^S)$, and write $\Lambda=UD^{\gamma_1}UD^{\gamma_2}\cdots UD^{\gamma_k}$ and $\Lambda'=UD^{\gamma_1'}UD^{\gamma_2'}\cdots UD^{\gamma_k'}$. Put $a_i=\gamma_{k-i+1}'$ and $b_i=\gamma_{k-i+1}$. The fact that $\Lambda$ and $\Lambda'$ are Dyck paths guarantees that $a_1+\cdots+a_k=b_1+\cdots+b_k=k$. The fact that $\Lambda\leq_S\Lambda'$ tells us that $a_{k-i+1}+\cdots+a_k\leq b_{k-i+1}+\cdots+b_k$ for all $i\in\{1,\ldots,k\}$. Appealing to Lemma \ref{Lem3}, we obtain a matrix $M=(m_{ij})$ satisfying the properties (i)--(iv) listed in the statement of that lemma. 
It follows from Remark \ref{Rem1} that we can use such a matrix to obtain a permutation $\pi\in S_n$ (where $n=2k+1$) with $M(\pi)=M$ and with the property that $\Lambda_i=D$ if and only if $n-i\in\Des(\pi)$ and $\Lambda_i'=U$ if and only if $i\in\Des(\rot(\pi))$. Because $D$ appears exactly $k$ times in $\Lambda$, the permutation $\pi$ has exactly $k$ descents. The construction of $\pi$ described in Remark \ref{Rem1} along with property (iv) from Lemma \ref{Lem3} guarantee that $\pi$ avoids $312$. To see that $\pi$ is uniquely sorted, it suffices by Theorem \ref{Thm1} and Proposition \ref{Prop1} to see that it has a canonical hook configuration. This follows from the fact that every prefix of $\Lambda$ contains at least as many copies of $U$ as copies of $D$. Indeed, if $d_1<\cdots<d_k$ are the descents of $\pi$, then this property of $\Lambda$ guarantees that the plot of $\pi$ has at least $\ell$ left-to-right maxima to the right of $(d_{k-\ell+1},\pi_{d_{k-\ell+1}})$ for every $\ell\in\{1,\ldots,k\}$. This means that it is always possible to find a northeast endpoint for the hook $H_{k-\ell+1}$ when we construct the canonical hook configuration of $\pi$. Consequently, $\pi\in\mathcal U_{2k+1}(312)$. Properties (ii) and (iii) from Lemma \ref{Lem3} ensure that $\DL_k(\pi)=(\Lambda,\Lambda')$. 

The injectivity of $\DL_k$ follows from Remark \ref{Rem1} and the uniqueness statement in Lemma \ref{Lem3}. Indeed, if $\DL_k(\pi)=(\Lambda,\Lambda')$, then $M(\pi)$ satisfies properties (i)--(iv) from the lemma with $a_i=\gamma_{k-i+1}'$ and $b_i=\gamma_{k-i+1}$ (it satisfies property (iv) because $\pi$ avoids $312$). According to Remark \ref{Rem1}, the matrix $M(\pi)$ uniquely determines $\pi$.
\end{proof}

Combining Theorem \ref{Thm2} with equation \eqref{Eq1} yields the following corollary. 

\begin{corollary}\label{Cor1}
For each nonnegative integer $k$, \[|\mathcal U_{2k+1}(312)|=C_kC_{k+2}-C_{k+1}^2=\frac{6}{(k+1)(k+2)^2(k+3)}{2k\choose k}{2k+2\choose k+1}.\]
\end{corollary}

\section{Tamari Intervals, $\mathcal U_{2k+1}(132)$, and $\mathcal U_{2k+1}(231)$}\label{Sec:Tamari}

In Section \ref{Sec:Operators}, we introduced sliding operators $\swu,\swd,\swl,\swr$. In the previous section, we found bijections $\DL_k:\mathcal U_{2k+1}(312)\to\Int(\mathcal L_k^S)$, where $\mathcal L_k^S$ is the $k^\text{th}$ Stanley lattice. Recall that $\mathcal L_k^T$ is the $k^\text{th}$ Tamari lattice, which we defined in Section \ref{Sec:LatticeBack}. The purpose of the current section is to show that for each nonnegative integer $k$, the maps $\swu:\mathcal U_{2k+1}(231)\to\mathcal U_{2k+1}(132)$ and $\DL_k\circ\swl:\mathcal U_{2k+1}(132)\to\Int(\mathcal L_k^T)$ are bijections.\footnote{It is not clear at this point that the composition $\DL_k\circ\swl:\mathcal U_{2k+1}(132)\to\Int(\mathcal L_k^T)$ is even well defined, but we will see that it is.} We have actually already done all of the heavy lifting needed to establish the first of these bijections. 

\begin{theorem}\label{Thm3}
For each nonnegative integer $k$, the maps $\swu:\mathcal U_{2k+1}(231)\to\mathcal U_{2k+1}(132)$ and $\swd:\mathcal U_{2k+1}(132)\to\mathcal U_{2k+1}(231)$ are bijections that are inverses of each other. 
\end{theorem}

\begin{proof}
Lemma \ref{Lem5} tells us that $\swu:\Av(231)\to\Av(132)$ and $\swd:\Av(132)\to\Av(231)$ are bijections that are inverses of each other. These maps also preserve lengths of permutations, so it suffices to show that they map uniquely sorted permutations to uniquely sorted permutations. If $\pi\in\mathcal U_{2k+1}(231)$, then we know from Theorem \ref{Thm1} and Lemma \ref{Lem1} that $\des(\swu(\pi))=\des(\pi)=k$. Lemma \ref{Lem6} tells us that $\swu(\pi)$ is sorted, so it follows from Theorem \ref{Thm1} that $\swu(\pi)$ is uniquely sorted. This shows that $\swu(\mathcal U_{2k+1}(231))\subseteq\mathcal U_{2k+1}(132)$, and a similar argument proves the reverse containment. 
\end{proof}

We now proceed to establish our bijections between $132$-avoiding uniquely sorted permutations and intervals in Tamari lattices. This essentially amounts to proving that if $\pi\in\mathcal U_{2k+1}(312)$, then $\swr(\pi)$ is sorted if and only if $\DL_k(\pi)\in\Int(\mathcal L_k^T)$. We do this in the next two propositions. 

\begin{proposition}\label{Prop2}
If $\pi\in\mathcal U_{2k+1}(312)$ is such that $\DL_k(\pi)\in\Int(\mathcal L_k^T)$, then $\swr(\pi)$ is sorted. 
\end{proposition}

\begin{proof}
We prove the contrapositive. Assume $\swr(\pi)$ is not sorted. Let $n=2k+1$. Since $\pi$ is sorted and $\swr(\pi)=\swr_1\circ\cdots\circ\swr_n(\pi)$ by definition, there exists $i\in[n]$ such that the permutation $\pi':=\swr_{i+1}\circ\cdots\circ\swr_n(\pi)$ is sorted while $\pi'':=\swr_i\circ\cdots\circ\swr_n(\pi)=\swr_i(\pi')$ is not. Write $\pi'=\pi_1'\cdots\pi_n'$ and $\pi''=\pi_1''\cdots\pi_n''$. Let $a=\pi_i'$, and let $\ell\in[n]$ be such that $\pi_\ell=a$. Because $\pi$ avoids $312$, its plot has the shape shown in Figure \ref{Fig10}. Indeed, all of the points to the northwest of $(\ell,a)$ (the points in $\lambda$) must appear to the right of the points to the southwest of $(\ell,a)$ (the points in $\mu$) since $a$ cannot form the last entry in a $312$ pattern in $\pi$. Similarly, the points in $\sigma$ must all appear above the points in $\lambda$ since $a$ cannot be the smallest entry in a $312$ pattern in $\pi$. Using the definitions of the maps $\swr_{i+1},\ldots,\swr_n$, we find that the shape of $\pi'$ is as shown in Figure \ref{Fig10}. For example, all of the points in $\mu'$ are higher than all of the points in $\tau'$ because there cannot be any points to the right of $(i,a)$ in the plot of $\pi'$ that form the rightmost entry in a $132$ pattern (by Remark \ref{Rem3}). Also, every time we apply one of the maps $\swr_{i+1},\ldots,\swr_n$, the points above $\lambda$ either end up to the right of $\lambda$ or end up to the left of all of the points in and below $\lambda$. Finally, it follows from the definition of $\swr_i$ that $\pi''$ has the shape shown in Figure \ref{Fig10}. The boxes in these diagrams are meant to represent places where there \emph{could} be points, but boxes could be empty. 

\begin{figure}[h]
\begin{center}
\includegraphics[width=\linewidth]{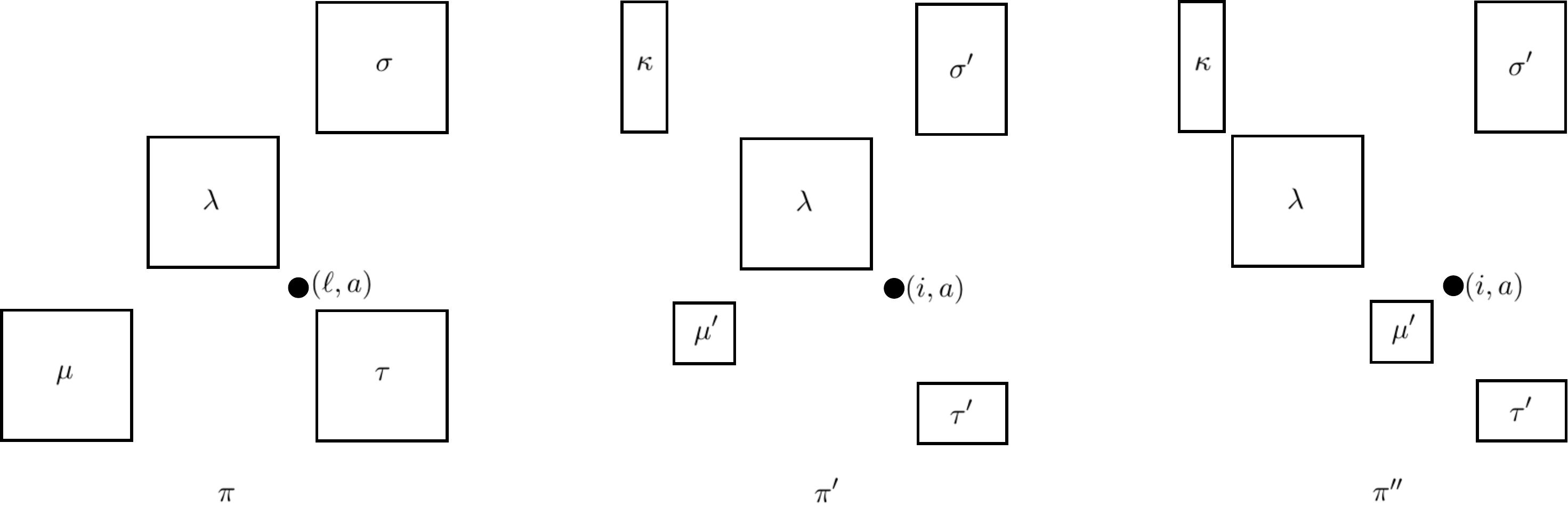}
\caption{The shapes of the plots of $\pi,\pi',\pi''$.}
\label{Fig10}
\end{center}  
\end{figure}

Because $\pi'$ is sorted, it has a canonical hook configuration $\mathcal H'$ by Proposition \ref{Prop1}. Let $Q(\lambda)$ (respectively, $Q(\sigma')$) be the number of hooks in $\mathcal H'$ with northeast endpoints in $\lambda$ (respectively, $\sigma'$) whose southwest endpoints are not in $\lambda$ (respectively, $\sigma'$). Recall the definition of the deficiency statistic $\de$ from the proof of Lemma \ref{Lem6}. Again, one can think of $\de(\xi)$ as the number of descent tops in the plot of $\xi$ that cannot find northeast endpoints within $\xi$ for their hooks. Every southwest endpoint counted by $\de(\lambda)$, $\de(\mu')$, or $\de(\kappa)$ belongs to a hook whose northeast endpoint is counted by either $Q(\lambda)$ or $Q(\sigma')$. To justify this, we need to see that $(i,a)$ is not the northeast endpoint of a hook in $\mathcal H'$. This follows from the second bullet in Remark \ref{Rem4} and the fact that $\lambda$ is nonempty (because $\pi'$ and $\pi''$ are distinct). Therefore, $Q(\sigma')+Q(\lambda)\geq \de(\lambda)+\de(\mu')+\de(\kappa)$. Hence, $Q(\lambda)\geq \de(\lambda)+\de(\mu')+\de(\kappa)-Q(\sigma')$. 

When we try to construct the canonical hook configuration of $\pi''$, we must fail at some point because $\pi''$ is not sorted. The plots of $\pi'$ and $\pi''$ are the same to the right of the point $(i,a)$, so this failure must occur when we try to find the northeast endpoint of a hook whose southwest endpoint is in $\lambda$, $\mu'$, or $\kappa$. All choices for these northeast endpoints are either $(i,a)$ or are in $\sigma'$, and the choices in $\sigma'$ contain all of the points in $\sigma'$ that are counted by $Q(\sigma')$. It follows that $Q(\sigma')+1<\de(\lambda)+\de(\mu')+\de(\kappa)$. Using the last line from the preceding paragraph, we find that $Q(\lambda)\geq \de(\lambda)+\de(\mu')+\de(\kappa)-Q(\sigma')>1$. Therefore, $Q(\lambda)\geq 2$. By the definition of $Q(\lambda)$, there are at least two points in $\lambda$ that are northeast endpoints of hooks in $\mathcal H'$ and whose southwest endpoints are not in $\lambda$. These points (after they have been slid horizontally) are still northeast endpoints of hooks in the canonical hook configuration $\mathcal H$ of $\pi$. Indeed, this is a consequence of Lemma \ref{Lem2} because $\pi$ is uniquely sorted and these points are left-to-right maxima of the plot of $\pi$. In the plot of $\pi$, the hooks with these two points as northeast endpoints must have southwest endpoints that are not in $\lambda$. 

Every time we mention hooks, southwest endpoints, or northeast endpoints in the remainder of the proof, we refer to those of the canonical hook configuration $\mathcal H$ of $\pi$. Note that $\pi_n=n$ because $\pi$ is sorted. Lemma \ref{Lem4} tells us that the left-to-right maxima of the plot of $\pi$ are $(1,\pi_1),\mathfrak N_1,\mathfrak N_2,\ldots,\mathfrak N_k$, where $\mathfrak N_1,\ldots,\mathfrak N_k$ are the northeast endpoints. Let $\mathfrak R_0,\ldots,\mathfrak R_k$ be these left-to-right maxima, written in order from right to left (so $\mathfrak R_0=(n,n)$ and $\mathfrak R_k=(1,\pi_1)$). Let $\DL_k(\pi)=(\Lambda,\Lambda')$, where $\Lambda=UD^{\gamma_1}UD^{\gamma_2}\cdots UD^{\gamma_k}$ and $\Lambda'=UD^{\gamma_1'}UD^{\gamma_2'}\cdots UD^{\gamma_k'}$. Because $\pi$ avoids $312$, the points lying horizontally between $\mathfrak R_{i-1}$ and $\mathfrak R_i$ are decreasing in height from left to right for every $i\in[k]$. Similarly, the points lying vertically between $\mathfrak R_i$ and $\mathfrak R_{i+1}$ are decreasing in height from left to right for every $i\in[k]$. This implies that $\gamma_i$ is the number of points lying horizontally between $\mathfrak R_{i-1}$ and $\mathfrak R_i$, while $\gamma_i'$ is the number of points lying vertically between $\mathfrak R_i$ and $\mathfrak R_{i+1}$ (we make the convention that all points ``lie above $\mathfrak R_{k+1}$," even though $\mathfrak R_{k+1}$ is not actually a point that we have defined).

We saw above that there are (at least) two points in $\lambda$ that are northeast endpoints of hooks whose southwest endpoints are not in $\lambda$. Their southwest endpoints must be in $\mu$. Since these points are northeast endpoints, they are $\mathfrak R_{j-1}$ and $\mathfrak R_{j+m-1}$ for some $j\in[k]$ and $m\geq 1$. We may assume that we have chosen these points as far left as possible. In particular, $\mathfrak R_{j+m-1}$ is the leftmost point in $\lambda$. 

For $r\in\{j,\ldots,j+m-1\}$, let $\zeta^{(r)}$ be the permutation whose plot is the portion of $\lambda$ obtained by deleting everything to the left of $\mathfrak R_r$ and everything equal to or to the right of $\mathfrak R_{j-1}$. Because none of the hooks with southwest endpoints in $\lambda$ have $\mathfrak R_{j-1}$ as their northeast endpoints, all of the southwest endpoints in $\zeta^{(r)}$ belong to hooks whose northeast endpoints are in $\zeta^{(r)}$. When $r=j$, this implies that there are no points horizontally between $\mathfrak R_{j-1}$ and $\mathfrak R_j$. Thus, $\gamma_j=0$. When $r=j+1$, this implies that there is at most one point other than $\mathfrak R_j$ that lies horizontally between $\mathfrak R_{j-1}$ and $\mathfrak R_{j+1}$. Thus, $\gamma_j+\gamma_{j+1}\leq 1$. Continuing in this way, we find that $\gamma_j+\cdots+\gamma_{j+v}\leq v$ for all $v\in\{0,\ldots,m-1\}$. Referring to Definition \ref{Def3}, we find that $\lon_j(\Lambda)\geq m$.  

Now recall that we chose the point $\mathfrak R_{j-1}$ as far left as possible subject to the conditions that it is not the leftmost point in $\lambda$ and that it is the northeast endpoint of a hook whose southwest endpoint is in $\mu$. When the northeast endpoint of this hook was determined in the canonical hook configuration construction, we did not choose any of the points $\mathfrak R_j,\ldots,\mathfrak R_{j+m-2}$. This means that we \emph{could not} have chosen any of these points, so they must have already been northeast endpoints of other hooks. These other hooks must all have their southwest endpoints in $\lambda$ (since we chose $\mathfrak R_{j-1}$ as far left as possible). These southwest endpoints are descent tops, and the corresponding descent bottoms are not left-to-right maxima. This tells us that there are at least $m-1$ points other than $\mathfrak R_{j+1},\ldots,\mathfrak R_{j+m-2}$ that lie horizontally between $\mathfrak R_j$ and $\mathfrak R_{j+m-1}$. All of these points are in $\lambda$, so they must lie above $\mathfrak R_{j+m}$ and below $\mathfrak R_{j+1}$. This forces $\gamma_{j+1}'+\cdots+\gamma_{j+m-1}'\geq m-1$. However, $(\ell,a)$ is another point that lies above $\mathfrak R_{j+m}$ and below $\mathfrak R_{j+1}$, so we actually have $\gamma_{j+1}'+\cdots+\gamma_{j+m-1}'>m-1$. Since $\gamma_j'\geq 0$, this means that $\gamma_j'+\cdots+\gamma_{j+m-1}'>m-1$. According to Definition \ref{Def3}, $\lon_j(\Lambda')\leq m-1$. We have seen that $\lon_j(\Lambda)\geq m$, so it is immediate from the definition of the Tamari lattice $\mathcal L_k^T$ that $(\Lambda,\Lambda')\not\in\Int(\mathcal L_k^T)$.  
\end{proof}

\begin{proposition}\label{Prop3}
If $\pi\in\mathcal U_{2k+1}(312)$ is such that $\swr(\pi)$ is sorted, then $\DL_k(\pi)\in\Int(\mathcal L_k^T)$. 
\end{proposition}

\begin{proof}
Let $\DL_k(\pi)=(\Lambda,\Lambda')$, where $\Lambda=UD^{\gamma_1}UD^{\gamma_2}\cdots UD^{\gamma_k}$ and $\Lambda'=UD^{\gamma_1'}UD^{\gamma_2'}\cdots UD^{\gamma_k'}$. We are going to prove the contrapositive of the proposition, so assume $(\Lambda,\Lambda')\not\in\Int(\mathcal L_k^T)$. This means that there exists $j\in[k]$ such that $\lon_j(\Lambda)>\lon_j(\Lambda')$. As in the proof of the previous proposition, we let $\mathfrak R_0,\ldots,\mathfrak R_k$ be the left-to-right maxima of the plot of $\pi$ listed in order from right to left. Let $(\ell,a)$ be the highest point in the plot of $\pi$ that appears to the southeast of $\mathfrak R_j$. Let $\mathfrak R_{j+m}$ be the leftmost point that is higher than $(\ell,a)$. Using the assumption that $\pi$ avoids $312$, we find that the plot of $\pi$ has the following shape\footnote{The figure is drawn to make it look as though $m\geq 1$, but it is possible to have $m=0$. In this case, $\lambda$ consists of the single point $\mathfrak R_j$.}: \[\includegraphics[width=.35\linewidth]{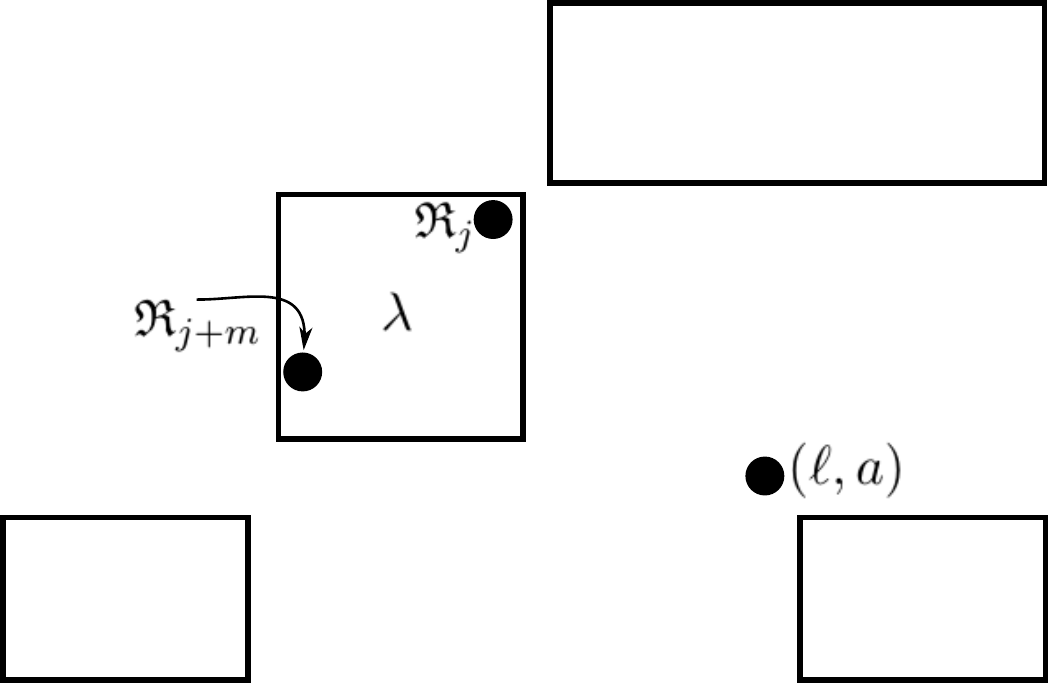}\] The image is meant to indicate that $\mathfrak R_j$ is the highest and the rightmost point in $\lambda$ and that $\mathfrak R_{j+m}$ is the leftmost point in $\lambda$. The boxes in this diagram represent places where there could be points, but boxes could be empty. 

Let $t=\lon_j(\Lambda)$ and $t'=\lon_j(\Lambda')$. Since $t>t'\geq 0$, it follows from Definition \ref{Def3} that $\gamma_j=0$ and 
\begin{equation}\label{Eq6}
\gamma_j+\cdots+\gamma_{j+t'}\leq t'<\gamma_j'+\cdots+\gamma_{j+t'}'.
\end{equation} 
Now, $\gamma_j+\cdots+\gamma_{j+t'}$ is the number of points in the plot of $\pi$ other than $\mathfrak R_j,\ldots,\mathfrak R_{j+t'-1}$ lying horizontally between $\mathfrak R_{j-1}$ and $\mathfrak R_{j+t'}$. Since $\gamma_j=0$, this is actually the same as the number of points in the plot of $\pi$ other than $\mathfrak R_{j+1},\ldots,\mathfrak R_{j+t'-1}$ lying horizontally between $\mathfrak R_j$ and $\mathfrak R_{j+t'}$. Similarly, $\gamma_j'+\cdots+\gamma_{j+t'}'$ is the number of points other than $\mathfrak R_{j+1},\ldots,\mathfrak R_{j+t'}$ lying vertically between $\mathfrak R_j$ and $\mathfrak R_{j+t'+1}$. If $t'\leq m-1$, then all of these points counted by $\gamma_j'+\cdots+\gamma_{j+t'}'$  are in $\lambda$. In fact, they all lie horizontally between $\mathfrak R_j$ and $\mathfrak R_{j+t'}$, so they are among the points counted by $\gamma_j+\cdots+\gamma_{j+t'}$. This contradicts \eqref{Eq6}, so we must have $t'\geq m$. This means that $t\geq m+1$, so it follows from Definition \ref{Def3} that $\gamma_j+\cdots+\gamma_{j+m}\leq m$. The points in the plot of $\pi$ other than $\mathfrak R_{j+1},\ldots,\mathfrak R_{j+m-1}$ that lie horizontally between $\mathfrak R_j$ and $\mathfrak R_{j+m}$ are all in $\lambda$. Letting $|\lambda|$ denote the number of points in $\lambda$, we find that  
\begin{equation}\label{Eq7}
|\lambda|=\gamma_j+\cdots+\gamma_{j+m}+m+1\leq 2m+1.
\end{equation}
The $m+1$ points $\mathfrak R_j,\ldots,\mathfrak R_{j+m}$, which lie in $\lambda$, are not descent bottoms in the plot of $\lambda$, so it follows from \eqref{Eq7} that $\des(\lambda)\leq|\lambda|-(m+1)\leq(|\lambda|-1)/2$. We know that $\lambda$ avoids $312$ because $\pi$ does, so we can use Lemma \ref{Lem1} to see that 
\begin{equation}\label{Eq8}
\des(\swr(\lambda))\leq(|\lambda|-1)/2.
\end{equation}

We now check that $\swr(\pi)$ has the following shape: \[\includegraphics[width=.35\linewidth]{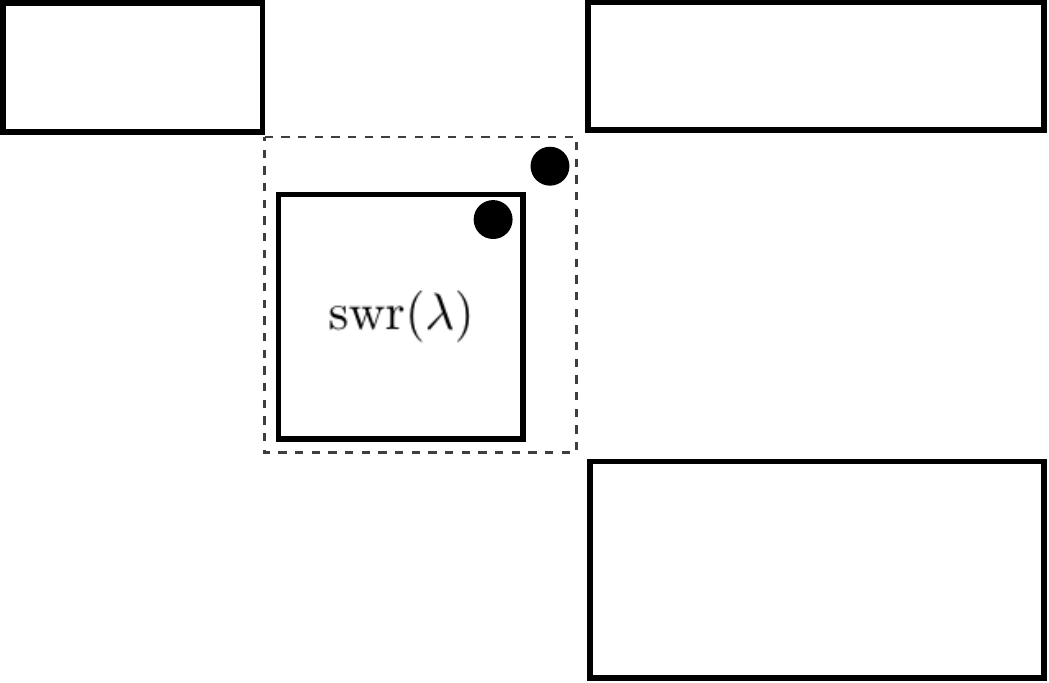}\] To see this, imagine applying the composition $\swr_1\circ\cdots\circ\swr_n$ to $\pi$. It is helpful to identify points with their heights and imagine sliding these points around horizontally. The points in the block $\lambda$ form a contiguous horizontal interval and a contiguous vertical interval in the plot of $\pi$, and none of the maps $\swr_i$ will change this fact. Note that these maps change the block from the plot of $\lambda$ to the plot of $\swr(\lambda)$. At some point during the sliding process, one of the maps $\swr_i$ will cause all of the points to the southwest of the point with height $a$ to move to the right of the points to the northwest of the point with height $a$. This causes all of the points southwest of the block to move to the southeast of the block. None of the maps $\swr_i$ can slide points below the block to the left of the block, which is why all of the points lower than the block in the plot of $\swr(\pi)$ appear to the southeast of the block. If $v$ is the height of $\mathfrak R_j$, then there is a point with height $v+1$ (because $j\geq 1$). The point with height $v+1$ appears to the right of $\mathfrak R_j$ in the plot of $\pi$, and none of the maps $\swr_i$ can move the point with height $v+1$ to the left of $\mathfrak R_j$. Furthermore, there are no points in the plot of $\swr(\pi)$ lying horizontally between a point in the block and the point with height $v+1$ because $\swr(\pi)$ avoids $132$ (by Lemma \ref{Lem5}). This is why the point with height $v+1$ is immediately to the right of the point with height $v$ in the plot of $\swr(\pi)$.

Because $\pi\in\mathcal U_{2k+1}(312)$, we know from Theorem \ref{Thm1} and Lemma \ref{Lem1} that $\des(\swr(\pi))=\des(\pi)=k$. Our goal is to show that $\swr(\pi)$ is not sorted, so suppose by way of contradiction that it is sorted. Theorem \ref{Thm1} tells us that $\swr(\pi)$ is uniquely sorted. Let $\mathfrak N_1',\ldots,\mathfrak N_k'$ be the northeast endpoints of the hooks in the canonical hook configuration of $\swr(\pi)$. Let $\DB(\swr(\pi))$ be the set of descent bottoms in the plot of $\swr(\pi)$. Let $\mathfrak Q=(1,\swr(\pi)_1)$ be the leftmost point in the plot of $\swr(\pi)$. According to Lemma \ref{Lem2}, the sets $\DB(\swr(\pi))$, $\{\mathfrak N_1',\ldots,\mathfrak N_k'\}$, and $\{\mathfrak Q\}$ form a partition of the set of points in the plot of $\swr(\pi)$. 

Now refer back to the above image of the plot of $\swr(\pi)$. Let $\mathcal E$ be the set of points lying in the dotted region. It is possible that there is a point in this plot lying to the northwest of all of the points in $\mathcal E$. It is also possible that there is no such point. In either case, 
\begin{equation}\label{Eq9}
|(\DB(\swr(\pi))\cup\{\mathfrak Q\})\cap\mathcal E|\leq\des(\swr(\lambda))+1.
\end{equation} Now note that there are no points in the plot of $\swr(\pi)$ lying to the southwest of the dotted region. This means that every element of $\{\mathfrak N_1',\ldots,\mathfrak N_k'\}\cap\mathcal E$ is the northeast endpoint of a hook (in the canonical hook configuration of $\swr(\pi)$) whose southwest endpoint is in the block labeled $\swr(\lambda)$. Hence, $|\{\mathfrak N_1',\ldots,\mathfrak N_k'\}\cap\mathcal E|$ is at most the number of southwest endpoints that lie in this block. Now recall from the first bullet in Remark \ref{Rem4} that the southwest endpoints of hooks are precisely the descent tops in the plot. It follows that $|\{\mathfrak N_1',\ldots,\mathfrak N_k'\}\cap\mathcal E|\leq\des(\swr(\lambda))$. Combining this observation with \eqref{Eq8} and \eqref{Eq9} yields \[|\mathcal E|=|(\DB(\swr(\pi))\cup\{\mathfrak Q\})\cap\mathcal E|+|\{\mathfrak N_1',\ldots,\mathfrak N_k'\}\cap\mathcal E|\leq\des(\swr(\lambda))+1+\des(\swr(\lambda))\leq|\lambda|.\] This is our desired contradiction because $|\mathcal E|=|\lambda|+1$. 
\end{proof}

\begin{theorem}\label{Thm4}
For each nonnegative integer $k$, the map $\DL_k\circ\swl:\mathcal U_{2k+1}(132)\to\Int(\mathcal L_k^T)$ is a bijection. 
\end{theorem}
\begin{proof}
First, recall from Lemma \ref{Lem5} that $\swl:\Av(132)\to\Av(312)$ and $\swr:\Av(312)\to\Av(132)$ are bijections that are inverses of each other. If $\pi\in\mathcal U_{2k+1}(132)$, then we know from Theorem \ref{Thm1} that $\pi$ is sorted and has $k$ descents. Lemmas \ref{Lem1} and \ref{Lem6} guarantee that $\swl(\pi)$ is sorted and has $k$ descents, so it follows from Theorem \ref{Thm1} that $\swl(\pi)\in\mathcal U_{2k+1}(312)$. This means that it actually makes sense to apply $\DL_k$ to $\swl(\pi)$. Since $\swr(\swl(\pi))=\pi$ is sorted, Proposition~\ref{Prop3} tells us that $\DL_k(\swl(\pi))\in\Int(\mathcal L_k^T)$. Hence, $\DL_k\circ\swl$ does indeed map $\mathcal U_{2k+1}(132)$ into $\Int(\mathcal L_k^T)$. The injectivity of the map $\DL_k\circ\swl:\mathcal U_{2k+1}(132)\to\Int(\mathcal L_k^T)$ follows from the injectivity of $\DL_k$ and the injectivity of $\swl$ on $\mathcal U_{2k+1}(132)$. To prove surjectivity, choose $(\Lambda,\Lambda')\in\Int(\mathcal L_k^T)$. Let $\sigma=\DL_k^{-1}(\Lambda,\Lambda')$. We know by the definition of $\DL_k$ that $\sigma\in\mathcal U_{2k+1}(312)$, so $\sigma$ has $k$ descents. According to Lemma~\ref{Lem1}, $\swr(\sigma)$ has $k$ descents. Since $\DL_k(\sigma)\in\Int(\mathcal L_k^T)$, it follows from Proposition~\ref{Prop2} that $\swr(\sigma)$ is sorted. By Theorem \ref{Thm1}, $\swr(\sigma)\in\mathcal U_{2k+1}(132)$. This proves surjectivity since $\DL_k\circ\swl(\swr(\sigma))=\DL_k(\sigma)=(\Lambda,\Lambda')$. 
\end{proof}

Combining Theorem \ref{Thm3}, Theorem \ref{Thm4}, and equation \eqref{Eq2} yields the following corollary. 

\begin{corollary}\label{Cor2}
For each nonnegative integer $k$, \[|\mathcal U_{2k+1}(132)|=|\mathcal U_{2k+1}(231)|=\frac{2}{(3k+1)(3k+2)}{4k+1\choose k+1}.\]
\end{corollary}

\section{Noncrossing Partition Intervals and $\mathcal U_{2k+1}(312,1342)$}\label{Sec:Kreweras}

We say two distinct blocks $B,B'$ in a set partition $\rho$ of $[k]$ form a \emph{crossing} if there exist $a,c\in B$ and $b,d\in B'$ such that either $a<b<c<d$ or $a>b>c>d$. A partition is \emph{noncrossing} if no two of its blocks form a crossing. Let $\NC_k$ be the set of noncrossing partitions of $[k]$ ordered by refinement. That is, $\rho\leq\kappa$ in $\NC_k$ if every block of $\rho$ is contained in a block of $\kappa$. 

As mentioned in the introduction, the Kreweras lattices $\mathcal L_k^K$ are isomorphic to the noncrossing partition lattices $\NC_k$ and have the Tamari lattices $\mathcal L_k^T$ as extensions. We want to find a bijection between uniquely sorted permutations avoiding $312$ and $1342$ and intervals in Kreweras (equivalently, noncrossing partition) lattices.
Since $\mathcal U_{2k+1}(312,1342)\subseteq \mathcal U_{2k+1}(312)$ and $\Int(\mathcal L_k^K)\subseteq\Int(\mathcal L_k^S)$, one might hope that the map $\DL_k$ from Section \ref{Sec:Stanley} would yield our desired bijection. In other words, it would be nice if we had $\DL_k(\mathcal U_{2k+1}(312,1342))=\Int(\mathcal L_k^K)$. This, however, is not the case. For example, $3254167\in\mathcal U_7(312,1342)$, but $\DL_3(3254167)=(UUDDUD,UUDUDD)\not\in\Int(\mathcal L_3^K)$ (see Figure \ref{Fig2}).
Therefore, we must define a different map. We find it convenient to only work with the noncrossing partition lattices in this section. Thus, our goal is to prove the following theorem. 

\begin{theorem}\label{Thm5}
For each nonnegative integer $k$, there is a bijection \[\Upsilon_k:\mathcal U_{2k+1}(312,1342)\to\Int(\NC_k).\]
\end{theorem}    
The proof of Theorem \ref{Thm5} in the specific case $k=0$ is trivial (we make the convention that  $\NC_0=\{\emptyset\}$), so we will assume throughout this section that $k\geq 1$. 

To prove this theorem, we make use of \emph{generating trees}, an enumerative tool that was introduced in \cite{Chung} and studied heavily afterward \cite{Banderier, West3, West2}. To describe a generating tree of a class of combinatorial objects, we first specify a scheme by which each object of size $n$ can be uniquely generated from an object of size $n-1$. We then label each object with the number of objects it generates. The generating tree consists of an ``axiom" that specifies the labels of the objects of size $1$ along with a ``rule" that describes the labels of the objects generated by each object with a given label. For example, in the generating tree 
\[\text{Axiom: }(2)\qquad\text{Rule: }(1)\leadsto(2),\quad(2)\leadsto(1)(2),\] the axiom $(2)$ tells us that we begin with a single object of size $1$ that has label $2$. The rule $(1)\leadsto(2),\hspace{.15cm}(2)\leadsto(1)(2)$ tells us that each object of size $n-1$ with label $1$ generates a single object of size $n$ with label $2$, whereas each object of size $n-1$ with label $2$ generates one object of size $n$ with label $1$ and one object of size $n$ with label $2$. This example generating tree is now classical (see Example 3 in \cite{West3}); it describes objects counted by the Fibonacci numbers. 

We are going to describe a generating tree for the class of intervals in noncrossing partition lattices and a generating tree for the class\footnote{Every combinatorial class has a ``size function." The ``size" of a permutation of length $2k+1$ in this class is $k$.} $\mathcal U(312,1342)$. We will find that there is a natural isomorphism between these two generating trees. This isomorphism yields the desired bijections $\Upsilon_k$. 

\begin{remark}\label{Rem2}
It is actually possible to give a short description of the bijection $\Upsilon_k$ that does not rely on generating trees. We do this in the next paragraph. However, it is not at all obvious from the definition we are about to give that this map is indeed a bijection from $\mathcal U_{2k+1}(312,1342)$ to $\Int(\NC_k)$. The current author was able to prove this directly, but the proof ended up being very long and tedious. For this reason, we will content ourselves with merely defining the map. We also omit the proof that this map is indeed the same as the map $\Upsilon_k$ that we will obtain later via generating trees, although this fact can be proven by tracing carefully through the relevant definitions. In order to avoid potential confusion arising from the fact that we have given different definitions of these maps and have not proven them to be equivalent, we use the symbol $\Upsilon_k'$ for the map defined in the next paragraph.   

Suppose we are given $\pi\in\mathcal U_{2k+1}(312,1342)$. Because $\pi$ is sorted, we know from Proposition \ref{Prop1} that it has a canonical hook configuration $\mathcal H$. Let $\mathfrak W_1,\ldots,\mathfrak W_k$ be the northeast endpoints of the hooks in $\mathcal H$ listed in increasing order of height. Let $\mathfrak U_\ell$ be the southwest endpoint of the hook whose northeast endpoint is $\mathfrak W_\ell$. The \emph{partner} of $\mathfrak W_\ell$, which we denote by $\mathfrak V_\ell$, is the point immediately to the right\footnote{We say a point $\mathfrak X$ is \emph{immediately to the right of} a point $\mathfrak X'$ if $\mathfrak X$ is the leftmost point to the right of $\mathfrak X'$. The phrases ``immediately to the left," ``immediately above," and ``immediately below" are defined similarly.} of $\mathfrak U_\ell$ in the plot of $\pi$. Let $\rho$ be the partition of $[k]$ obtained as follows. Place numbers $\ell,m\in [k]$ in the same block of $\rho$ if $\mathfrak V_\ell$ appears immediately above and immediately to the left of $\mathfrak V_m$ in the plot of $\pi$. Then, close all of these blocks by transitivity. Let $\kappa$ be the partition of $[k]$ obtained as follows. Place numbers $\ell,m\in[k]$ in the same block of $\kappa$ if they are in the same block of $\rho$ or if $\mathfrak W_\ell$ appears immediately above and immediately to the left of $\mathfrak V_m$ in the plot of $\pi$. Then, close all of these blocks by transitivity. Let $\Upsilon_k'(\pi)=(\rho,\kappa)$. Figure \ref{Fig11} shows an example application of each of the maps $\Upsilon_1',\Upsilon_2',\Upsilon_3',\Upsilon_4'$ (which are secretly the same as the maps $\Upsilon_1,\Upsilon_2,\Upsilon_3,\Upsilon_4$ defined later). At this point in time, the reader should ignore the horizontal maps, the green arrows, and the green shading in Figure \ref{Fig11}. \hspace*{\fill}$\lozenge$
\end{remark}

We now proceed to describe the generating tree for the combinatorial class of intervals in noncrossing partition lattices. Let us say an interval $(\rho,\kappa)\in\Int(\NC_k)$ \emph{generates} an interval $(\widetilde\rho,\widetilde\kappa)\in\Int(\NC_{k+1})$ if $\rho$ and $\kappa$ are the partitions obtained by removing the number $k+1$ from its blocks in $\widetilde\rho$ and $\widetilde\kappa$, respectively. We say a block $B$ of a noncrossing partition $\rho$ is \emph{exposed} if there is no block $B'\in\rho$ with $\min B'<\min B\leq \max B<\max B'$ (i.e., there is no block above $B$ in the arch diagram of $\rho$). Given $(\rho,\kappa)\in\Int(\NC_k)$, let $E_1(\rho,\kappa)$ be the set of exposed blocks of $\kappa$. Let $E_2(\rho,\kappa)$ be the set of exposed blocks of $\rho$ that are contained in exposed blocks of $\kappa$. We can write $E_1(\rho,\kappa)=\{\mathscr B_1,\ldots,\mathscr B_t\}$, where $\mathscr B_1,\ldots,\mathscr B_t$ are ordered from right to left. For each $i\in\{1,\ldots,t\}$, let $\mathscr B_{i,1},\ldots,\mathscr B_{i,m_i}$ be the blocks of $\rho$ that are contained in $\mathscr B_i$, ordered from right to left. We have $E_2(\rho,\kappa)=\{\mathscr B_{i,j}:1\leq i\leq t\text{ and }1\leq j\leq m_i\}$. Define the label of $(\rho,\kappa)$ to be $a(\rho,\kappa)=1+|E_1(\rho,\kappa)|+|E_2(\rho,\kappa)|=1+t+\sum_{i=1}^t m_i$.  

There are $a(\rho,\kappa)$ different operations that generate an interval in $\Int(\NC_{k+1})$ from the interval $(\rho,\kappa)\in\Int(\NC_k)$; we call these operations $\mathfrak u$, $\mathfrak v_i$ (for $1\leq i\leq t$), and $\mathfrak w_{i,j}$ (for $1\leq i\leq t$ and $1\leq j\leq m_i$). First, define $\mathfrak u(\rho,\kappa)$ to be the interval whose first and second partitions are obtained
by appending the singleton block $\{k+1\}$ to $\rho$ and $\kappa$, respectively. This has the effect of increasing the value of the label by $2$, so $a(\mathfrak u(\rho,\kappa))=a(\rho,\kappa)+2$. Let $\mathfrak v_i(\rho,\kappa)$ be the interval whose second partition is obtained by adding the number $k+1$ to the exposed block $\mathscr B_i$ of $\kappa$ and whose first partition is obtained by appending the singleton block $\{k+1\}$ to $\rho$. Note that $a(\mathfrak v_i(\rho,\kappa))=a(\rho,\kappa)+1-(i-1)-\sum_{r=1}^{i-1}m_r$. Finally, let $\mathfrak w_{i,j}(\rho,\kappa)$ be the interval whose second partition is obtained by adding the number $k+1$ to the exposed block $\mathscr B_i$ of $\kappa$ and whose first partition is obtained by adding the number $k+1$ to the exposed block $\mathscr B_{i,j}$ of $\rho$. Note that $a(\mathfrak w_{i,j}(\rho,\kappa))=a(\rho,\kappa)-(i-1)-\sum_{r=1}^{i-1}m_r-(j-1)$. The bottom of Figure \ref{Fig11} depicts some of these operations.

\begin{figure}[h]
\begin{center}
\includegraphics[width=.7\linewidth]{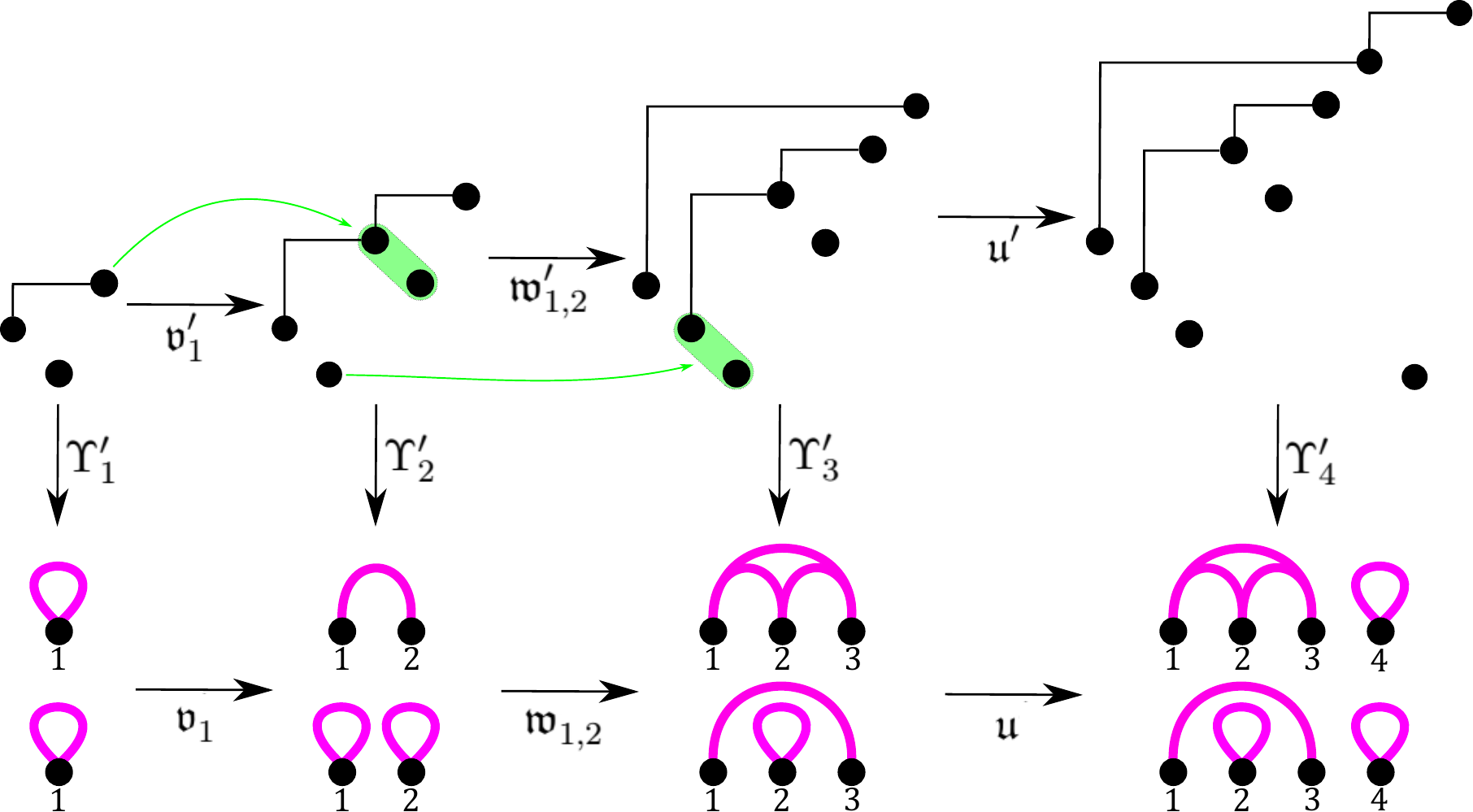}
\caption{Examples illustrating the various maps and operations defined in this section. We find it convenient to draw the plots of uniquely sorted permutations with their canonical hook configurations. The green arrows indicate the points that split into two when we apply $\mathfrak v_1'$ and $\mathfrak w_{1,2}'$. We have drawn each noncrossing partition interval with its first partition directly below its second partition.}
\label{Fig11}
\end{center}  
\end{figure}

It is now straightforward to check that for every integer $p$ with $3-a(\rho,\kappa)\leq p\leq 2$, there is a unique operation that generates from $(\rho,\kappa)$ an interval $(\widetilde\rho,\widetilde\kappa)$ with $a(\widetilde\rho,\widetilde\kappa)=a(\rho,\kappa)+p$. Thus, a generating tree of the class of intervals in noncrossing partition lattices is 
\begin{equation}\label{Eq13}
\text{Axiom: }(3)\qquad\text{Rule: }(m)\leadsto(3)(4)\cdots(m+2)\quad\text{for every }m\in\mathbb N.
\end{equation}
Let us remark that it is proven in \cite{Banderier} that the generating tree in \eqref{Eq13} describes objects counted by the $3$-Catalan numbers $\frac{1}{2k+1}{3k\choose k}$. Thus, we have actually reproven equation \eqref{Eq3}.  

We now want to describe a generating tree for the combinatorial class $\mathcal U(312,1342)$. We associate such permutations with their canonical hook configurations. Suppose $\pi=\pi_1\cdots\pi_{2k+1}\in\mathcal U_{2k+1}(312,1342)$, and let $\mathfrak P(i)$ be the point $(i,\pi_i)$ in the plot of $\pi$. We claim that there is a chain of hooks connecting the point $\mathfrak P(1)$ to the point $\mathfrak P(2k+1)$. We call this chain of hooks (including the endpoints of the hooks in the chain) the \emph{skyline} of $\pi$. For example, if $\pi=432657819$ is the permutation in the top right of Figure \ref{Fig11}, then the skyline of $\pi$ contains the points $\mathfrak P(1),\mathfrak P(7),\mathfrak P(9)$. To prove the claim, let $\mathfrak P(q)$ be the rightmost point that is connected to $\mathfrak P(1)$ via a chain of hooks. Suppose instead that $q<2k+1$. By the maximality of $q$, the point $\mathfrak P(q)$ is not the southwest endpoint of a hook. This means that it is not a descent top, so $\mathfrak P(q+1)$ is not a descent bottom. By Lemma \ref{Lem2}, $\mathfrak P(q+1)$ is the northeast endpoint of a hook. This hook must lie above the chain of hooks connecting $\mathfrak P(1)$ and $\mathfrak P(q)$ since hooks cannot cross or overlap each other (see Remark \ref{Rem4}). Consequently, the hook with northeast endpoint $\mathfrak P(q+1)$ must have a southwest endpoint lying to the left of $\mathfrak P(1)$. This is clearly impossible, so we have proven the claim.

We say a point $\mathfrak P(q)$ in the skyline is \emph{conjoined} if there is a point immediately below and immediately to the right of it (i.e., $\pi_q=\pi_{q+1}+1$). Otherwise, $\mathfrak P(q)$ is \emph{nonconjoined}. The point $\mathfrak P(2k+1)$ is nonconjoined because there is no point to its right. The point $\mathfrak P(1)$ is conjoined since, if it were not, the entries $\pi_1,\pi_2,\pi_1-1$ would form a $312$ pattern. We say a hook in the skyline is conjoined (respectively, nonconjoined) if its northeast endpoint is conjoined (respectively, nonconjoined). 

Recall from Remark~\ref{Rem2} that in a uniquely sorted permutation, the \emph{partner} of the northeast endpoint of a hook $H$ in the canonical hook configuration is the point immediately to the right of the southwest endpoint of $H$. Let us say a permutation $\pi\in\mathcal U_{2k+1}(312,1342)$ \emph{generates} a permutation $\widetilde\pi\in\mathcal U_{2k+3}(312,1342)$ if the plot of $\pi$ is obtained by removing the highest point in the plot of $\widetilde \pi$ (which is also the rightmost point and is also a northeast endpoint of a hook in the canonical hook configuration of $\widetilde\pi$) and the partner of that point from the plot of $\widetilde\pi$ (and then normalizing). Let $F_1(\pi)$ be the set of nonconjoined hooks in the skyline of $\pi$. Let $F_2(\pi)$ be the set of all hooks in the skyline of $\pi$. We can write $F_1(\pi)=\{{\mathscr H}_1,\ldots,{\mathscr H}_t\}$, where ${\mathscr H}_1,\ldots,{\mathscr H}_t$ are ordered from right to left. For each $i\in\{1,\ldots,t\}$, let ${\mathscr H}_{i,1},\ldots,{\mathscr H}_{i,m_i}$ be the hooks, ordered from right to left, that lie between $\mathscr H_i$ and $\mathscr H_{i+1}$ in the skyline of $\pi$, including $\mathscr H_i$ but not including $\mathscr H_{i+1}$. When $i=t$, these are the hooks that are equal to or to the left of $\mathscr H_t$ in the skyline. Note that ${\mathscr H_i}={\mathscr H}_{i,1}$. We have $F_2(\pi)=\{{\mathscr H}_{i,j}:1\leq i\leq t\text{ and }1\leq j\leq m_i\}$. Define the label of $\pi$ to be $b(\pi)=1+|F_1(\pi)|+|F_2(\pi)|=1+t+\sum_{i=1}^t m_i$.    

We now define operations $\mathfrak u'$, $\mathfrak v_i'$ (for $1\leq i\leq t$), and $\mathfrak w_{i,j}'$ (for $1\leq i\leq t$ and $1\leq j\leq m_i$), which generate a permutation in $\mathcal U_{2k+3}(312,1342)$ from the permutation $\pi\in\mathcal U_{2k+1}(312,1342)$ (the justification that the resulting permutation is in $\mathcal U_{2k+3}(312,1342)$ is given in Lemma~\ref{Lem7}). To define these operations, we first need to establish some notation. Given a point $\mathfrak X$ in the plot $\pi$, let $\spli_{\mathfrak X}(\pi)$ be the permutation whose plot is obtained by inserting a new point immediately below and immediately to the right of $\mathfrak X$ and then normalizing. In other words, we ``split" the point $\mathfrak X$ into two points in such a way that one of the new points is below and to the right of the other. For $1\leq i\leq t$, let $\mathfrak X_i$ be the northeast endpoint of the nonconjoined skyline hook $\mathscr H_i$. For $1\leq i\leq t$ and $1\leq j\leq m_i$, let $\mathfrak Y_{i,j}$ be the point immediately to the right of the southwest endpoint of $\mathscr H_{i,j}$ (in other words, $\mathfrak Y_{i,j}$ is the partner of the northeast endpoint of $\mathscr H_{i,j}$). We define 
\begin{equation}\label{Eq18}
\mathfrak u'(\pi)=(\pi\ominus 1)\oplus 1,\quad \mathfrak v_i'(\pi)=(\spli_{\mathfrak X_i}(\pi))\oplus 1,\quad\text{and}\quad\mathfrak w_{i,j}'(\pi)=(\spli_{\mathfrak Y_{i,j}}(\pi))\oplus 1.
\end{equation}

\begin{figure}[h]
\begin{center}
\includegraphics[width=.63\linewidth]{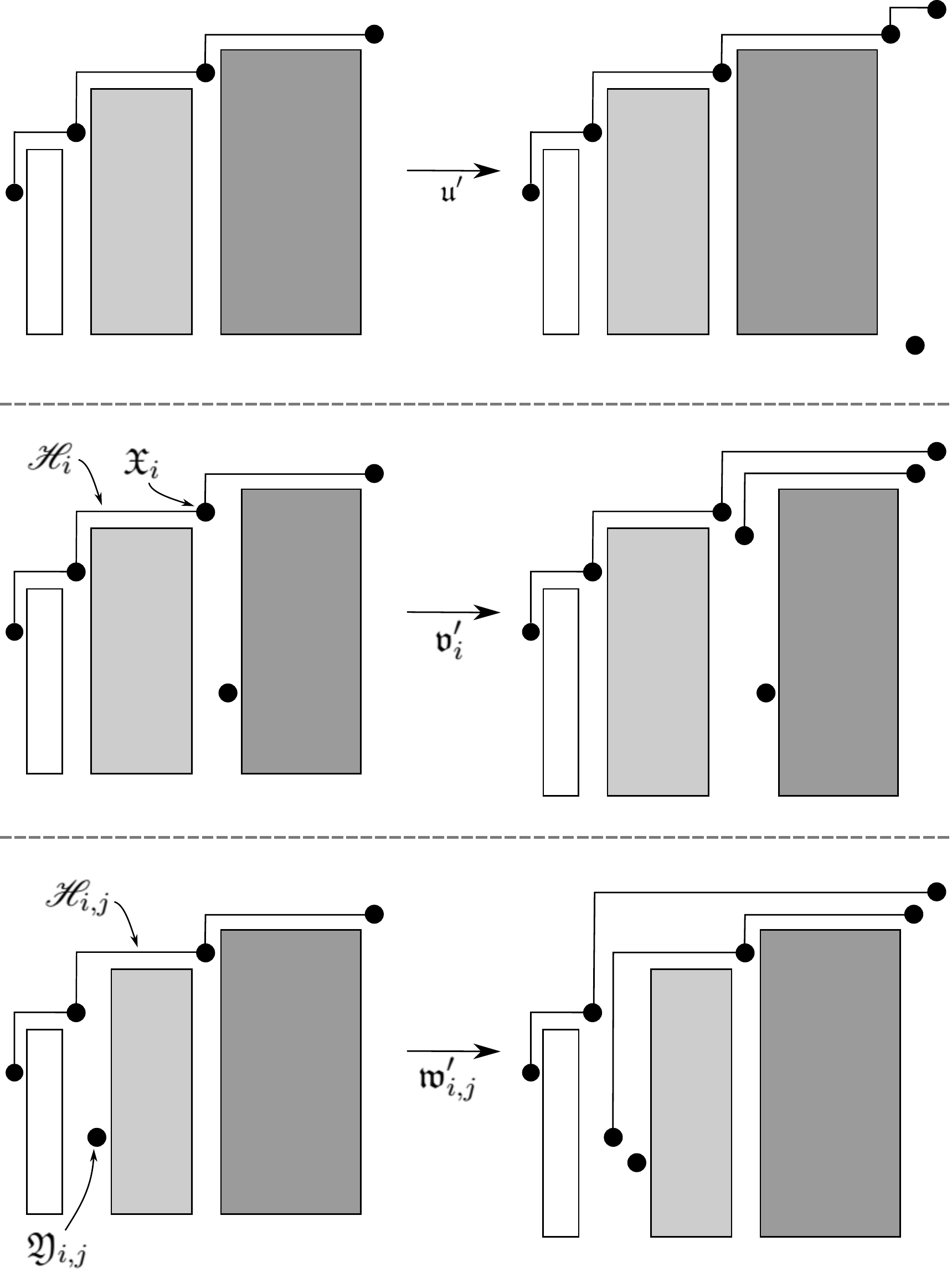}
\caption{Applications of the operations $\mathfrak u'$, $\mathfrak v_i'$, and $\mathfrak w_{i,j}'$. Within each of the three rows, the two blocks with the same amount of shading are meant to represent parts of the plot that have the same relative order as each other. The diagram also shows how the canonical hook configurations of the permutations change when we apply the operations (or when we undo the operations), as described in the proofs of Lemmas \ref{Lem8} and \ref{Lem7}.}
\label{Fig14}
\end{center}  
\end{figure}

In order to better understand these operations, we prove the following lemmas. The reader may find it helpful to refer to Figure \ref{Fig14} and the top of Figure \ref{Fig11} while reading the proofs that follow.

\begin{lemma}\label{Lem8}
Every permutation $\widetilde\pi\in\mathcal U_{2k+3}(312,1342)$ is generated by a unique permutation $\pi\in\mathcal U_{2k+1}(312,1342)$. Furthermore, there is an operation of $\pi$ from the list in \eqref{Eq18} that sends $\pi$ to $\widetilde\pi$.
\end{lemma} 

\begin{proof}
Let $\pi$ be the permutation that generates $\widetilde\pi$ (it is clearly unique). We first want to show that $\pi\in\mathcal U_{2k+1}(312,1342)$. Let $\mathfrak P(i)$ and $\widetilde{\mathfrak P}(i)$ denote the points $(i,\pi_i)$ and $(i,\widetilde \pi_i)$ in the plots of $\pi$ and $\widetilde\pi$, respectively. Since $\widetilde\pi$ is sorted, it has a canonical hook configuration $\widetilde{\mathcal H}$. The point $\widetilde{\mathfrak P}(2k+3)$ is the northeast endpoint of a hook $H$ in $\widetilde{\mathcal H}$.
Let $\widetilde{\mathfrak P}(\ell-1)$ be the southwest endpoint of $H$ so that $\widetilde{\mathfrak P}(\ell)$ is the partner of $\widetilde{\mathfrak P}(2k+3)$. The plot of $\pi$ is obtained from the plot of $\widetilde\pi$ by removing $\widetilde{\mathfrak P}(2k+3)$ and $\widetilde{\mathfrak P}(\ell)$ and normalizing. Note that $\pi$ avoids $312$ and $1342$ because $\widetilde\pi$ does. We know that $\widetilde{\mathfrak P}(2k+3)$ is not a descent bottom of the plot of $\widetilde \pi$ and that $\widetilde{\mathfrak P}(\ell)$ is a descent bottom. Using the fact that $\mathfrak P(\ell-1),\mathfrak P(\ell),\mathfrak P(\ell+1)$ do not form a $312$ pattern, one can check that $\des(\pi)=\des(\widetilde \pi)-1$. Theorem~\ref{Thm1} tells us that $\des(\widetilde\pi)=k+1$, so $\des(\pi)=k$. That same theorem now tells us that in order to prove $\pi\in\mathcal U_{2k+1}(312,1342)$, it suffices to prove $\pi$ is sorted. By Proposition \ref{Prop1}, we need to show that $\pi$ has a canonical hook configuration. 

Suppose first that $\ell$ is not a descent of $\widetilde\pi$. Since $\widetilde\pi_{\ell-1}>\widetilde\pi_\ell$, it follows from the fact that $\widetilde\pi$ avoids $312$ that $\widetilde\pi_{\ell+1}>\widetilde\pi_{\ell-1}$. Referring to the canonical hook configuration construction, we find that the hook with southwest endpoint $\widetilde{\mathfrak P}(\ell-1)$ must have $\widetilde{\mathfrak P}(\ell+1)$ as its northeast endpoint. This hook is $H$, and its northeast endpoint is $\widetilde{\mathfrak P}(2k+3)$. This shows that $\ell=2k+2$. The canonical hook configuration of $\pi$ is now obtained by removing the points $\widetilde{\mathfrak P}(2k+3)$ and $\widetilde{\mathfrak P}(\ell)=\widetilde{\mathfrak P}(2k+2)$ along with the hook $H$ and then normalizing. If $2\leq\widetilde\pi_{\ell}\leq\widetilde\pi_{\ell-1}-2$, then either the entries $\widetilde\pi_{\ell-1}-1,1,\widetilde\pi_\ell$ form a $312$ pattern or the entries $1,\widetilde\pi_{\ell-1}-1,\widetilde\pi_{\ell-1},\widetilde\pi_{\ell}$ form a $1342$ pattern. These are both impossible, so we must either have $\widetilde\pi_{\ell}=1$ or $\widetilde\pi_{\ell}=\widetilde\pi_{\ell-1}-1$. In the first case, $\mathfrak u'(\pi)=\widetilde\pi$. In the second case, $\mathfrak v_1'(\pi)=\widetilde\pi$. 

Next, assume $\ell$ is a descent of $\widetilde \pi$. In this case, $\widetilde{\mathfrak P}(\ell)$ is a descent top of the plot of $\widetilde\pi$, so it is the southwest endpoint of a hook $H'$ in $\widetilde{\mathcal H}$. Because $\mathfrak P(\ell-1)$, $\mathfrak P(\ell)$, and the northeast endpoint of $H'$ cannot form a $312$ pattern, $\mathfrak P(\ell-1)$ must be below and to the left of the northeast endpoint of $H'$. Therefore, we can draw a new hook $H''$ whose southwest endpoint is $\widetilde{\mathfrak P}(\ell-1)$ (which is also the southwest endpoint of $H$) and whose northeast endpoint is the northeast endpoint of $H'$. If we now remove the points $\widetilde{\mathfrak P}(2k+3)$ and $\widetilde{\mathfrak P}(\ell)$ along with the hooks $H$ and $H'$ (but keep the hook $H''$) and then normalize, we obtain the canonical hook configuration of $\pi$. To see an example of this, consider the permutations $\pi=21435$ and $\widetilde\pi=3215467$, which appear in the top of Figure \ref{Fig11} ($\ell=2$ in this example). This construction is also depicted in each of the lower two rows in Figure~\ref{Fig14}, where $\widetilde\pi$ is the permutation on the right and $\pi$ is the permutation on the left. This completes the proof that $\pi$ is sorted. Notice also that the point $\widetilde{\mathfrak P}(\ell-1)$ is in the skyline of $\widetilde\pi$. The chain of hooks connecting $\widetilde{\mathfrak P}(1)$ to $\widetilde{\mathfrak P}(\ell-1)$ in the plot of $\widetilde \pi$ still connects $\mathfrak P(1)$ to $\mathfrak P(\ell-1)$ in the plot of $\pi$, so $\mathfrak P(\ell-1)$ is in the skyline of $\pi$.

We still need to show that there is an operation from \eqref{Eq18} that sends $\pi$ to $\widetilde\pi$ when $\ell$ is a descent of $\widetilde\pi$. In this case, $\ell-1$ is a descent of $\pi$, and $\mathfrak P(\ell-1)$ is in the skyline of $\pi$. We have two cases to consider. First, assume $\widetilde{\mathfrak P}(\ell)$ lies immediately above $\widetilde{\mathfrak P}(\ell+1)$ in the plot of $\widetilde\pi$. In this case, let $\mathscr H_{i,j}$ be the hook in $\pi$ with southwest endpoint $\mathfrak P(\ell-1)$ (so $\mathfrak P(\ell)=\mathfrak Y_{i,j}$). We find that $\mathfrak w_{i,j}'(\pi)=\widetilde\pi$. For the second case, assume $\widetilde{\mathfrak P}(\ell)$ does not lie immediately above $\widetilde{\mathfrak P}(\ell+1)$ in the plot of $\widetilde\pi$. There exists some index $\delta$ such that $\widetilde\pi_{\ell+1}<\widetilde\pi_\delta<\widetilde\pi_\ell$. This implies that $\mathfrak P(\ell-1)$ is a nonconjoined point in the skyline of $\pi$, so it is the northeast endpoint of some nonconjoined hook $\mathscr H_i$ (that is, $\mathfrak P(\ell-1)=\mathfrak X_i$). Because $\widetilde\pi$ avoids $312$, we must have $\delta\leq \ell-2$. If there were some index $\delta'$ with $\widetilde\pi_\ell<\widetilde\pi_{\delta'}<\widetilde\pi_{\ell-1}$, then either the entries $\widetilde\pi_{\delta'},\widetilde\pi_{\delta},\widetilde\pi_\ell$ would form a $312$ pattern or the entries $\widetilde\pi_\delta,\widetilde\pi_{\delta'},\widetilde\pi_{\ell-1},\widetilde\pi_\ell$ would form a $1342$ pattern. This is impossible, so $\widetilde{\mathfrak P}(\ell-1)$ must be immediately above and to the left of $\widetilde{\mathfrak P}(\ell)$ in the plot of $\widetilde\pi$. Therefore, $\mathfrak v_i'(\pi)=\widetilde\pi$. We should check that $\ell-1\geq 2$ in this case so that $\mathfrak P(\ell-1)$ is actually the northeast endpoint of a hook. This follows from the fact, which we remarked earlier, that the point $\mathfrak P(1)$ is conjoined (because $\pi$ avoids $312$). 
\end{proof}

\begin{lemma}\label{Lem7}
If $\pi\in\mathcal U_{2k+1}(312,1342)$ and $\mathfrak o'$ is an operation in the list \eqref{Eq18}, then the permutation $\widetilde\pi=\mathfrak o'(\pi)$ is in $\mathcal U_{2k+3}(312,1342)$. Furthermore, $\mathfrak o'$ is the only operation in \eqref{Eq18} that sends $\pi$ to $\widetilde\pi$.
\end{lemma}
\begin{proof}
Let $\mathcal H$ be the canonical hook configuration of $\pi$. Let $\mathfrak P(i)$ and $\widetilde{\mathfrak P}(i)$ denote the points $(i,\pi_i)$ and $(i,\widetilde\pi_i)$, respectively. Since $\pi$ has $k$ descents and avoids $312$ and $1342$, it is not difficult to check that $\widetilde\pi$ has $k+1$ descents and avoids those same patterns. According to Theorem \ref{Thm1}, we need to show that $\widetilde \pi$ is sorted. By Proposition \ref{Prop1}, this amounts to showing that $\widetilde\pi$ has a canonical hook configuration $\widetilde{\mathcal H}$. To do this, we simply reverse the process described in the proof of Lemma \ref{Lem8} that allowed us to obtain the canonical hook configuration of $\pi$ from that of $\widetilde\pi$. More precisely, if $\mathfrak o'$ is $\mathfrak u'$ or $\mathfrak v_1'$, then we keep all of the hooks from $\mathcal H$ the same (modulo normalization of the plot) and attach a new hook with southwest endpoint $\widetilde{\mathfrak P}(2k+1)$ and northeast endpoint $\widetilde{\mathfrak P}(2k+3)$. Otherwise, we have either $\mathfrak o'=\mathfrak v_{i+1}'$ or $\mathfrak o'=\mathfrak w_{i,j}'$ for some $1\leq i\leq t$ and $1\leq j\leq m_t$. If $\mathfrak o'=\mathfrak v_{i+1}'$, let $\ell$ be such that $\mathfrak P(\ell-1)=\mathfrak X_{i+1}$. If $\mathfrak o'=\mathfrak w_{i,j}'$, let $\ell$ be such that $\mathfrak P(\ell)=\mathfrak Y_{i,j}$. In either case, $\ell-1$ is a descent of both $\pi$ and $\widetilde\pi$. Also, $\ell$ is a descent of $\widetilde\pi$. We obtain a (not canonical) configuration of hooks of $\widetilde\pi$ by keeping all of the hooks in $\mathcal H$ unchanged (modulo normalization). Let $H''$ be the hook in this configuration with southwest endpoint $\widetilde{\mathfrak P}(\ell-1)$. Now form a new hook $H$ of $\widetilde\pi$ with southwest endpoint $\widetilde{\mathfrak P}(\ell-1)$ and northeast endpoint $\widetilde{\mathfrak P}(2k+3)$. Form another new hook $H'$ of $\widetilde\pi$ whose southwest endpoint is $\widetilde{\mathfrak P}(\ell)$ and whose northeast endpoint is the northeast endpoint of $H''$. Removing the hook $H''$, we obtain the canonical hook configuration $\widetilde{\mathcal H}$ of $\widetilde\pi$. 

We need to show that there is at most one (equivalently, exactly one) operation from \eqref{Eq18} that sends $\pi$ to $\widetilde\pi$. As before, let $\widetilde{\mathfrak P}(\ell-1)$ be the southwest endpoint of the hook in $\widetilde{\mathcal H}$ whose northeast endpoint is $\widetilde{\mathfrak P}(2k+3)$. If $\ell-1=2k+1$, then the only such operations that are possible are $\mathfrak u'$ and $\mathfrak v_1'$. The permutations $\mathfrak u'(\pi)$ and $\mathfrak v_1'(\pi)$ are distinct, so there is at most one operation that sends $\pi$ to $\widetilde\pi$. Now suppose $\ell-1\neq 2k+1$. Appealing to the construction in the preceding paragraph, we find that $\mathfrak P(\ell-1)$ is in the skyline of $\pi$, so we can let $\mathscr H_{i,j}$ be the hook with southwest endpoint $\mathfrak P(\ell-1)$ (meaning $\mathfrak P(\ell)=\mathfrak Y_{i,j}$). We must have either $\mathfrak o'=\mathfrak v_{i+1}'$ (meaning $\mathfrak P(\ell-1)=\mathfrak X_{i+1}$ is nonconjoined) or $\mathfrak o'=\mathfrak w_{i,j}'$. If $\mathfrak P(\ell-1)$ is conjoined, then the only possibility is $\mathfrak w_{i,j}'$. If $\mathfrak P(\ell)$ is nonconjoined, then the permutations $\mathfrak v_i'(\pi)$ and $\mathfrak w_{i,j}'(\pi)$ are distinct. In this case, there is again only one operation that sends $\pi$ to $\widetilde\pi$. 
\end{proof}

We are now in a position to describe our generating tree. Lemmas \ref{Lem8} and \ref{Lem7} tell us that every permutation in $\mathcal U_{2k+3}(312,1342)$ is generated by a unique permutation in $\mathcal U_{2k+1}(312,1342)$ and that the number of permutations in $\mathcal U_{2k+3}(312,1342)$ that a permutation $\pi\in\mathcal U_{2k+1}(312,1342)$ generates is the number of operations listed in \eqref{Eq18}. The number of operations in \eqref{Eq18} is precisely the value of the label $b(\pi)$. In the proof of Lemma \ref{Lem7}, we described the procedure that produces the canonical hook configuration of $\widetilde\pi$ from that of $\pi$ when $\widetilde\pi$ is a permutation that $\pi$ generates. Tracing through this procedure, we find the following: 
\begin{itemize}
\item $|F_1(\mathfrak u'(\pi))|=|F_1(\pi)|+1$ and $|F_2(\mathfrak u'(\pi))|=|F_2(\pi)|+1$;
\item $|F_1(\mathfrak v_i'(\pi))|=|F_1(\pi)|-(i-1)$ and $|F_2(\mathfrak v_i'(\pi))|=|F_2(\pi)|+1-\sum_{r=1}^{i-1}m_r$; 
\item $|F_1(\mathfrak w_{i,j}'(\pi))|=|F_1(\pi)|-(i-1)$ and $|F_2(\mathfrak w_{i,j}'(\pi))|=|F_2(\pi)|-\sum_{r=1}^{i-1}m_r-(j-1)$. 
\end{itemize} Hence, we have $b(\mathfrak u'(\pi))=b(\pi)+2$, $b(\mathfrak v_i'(\pi))=b(\pi)+1-(i-1)-\sum_{r=1}^{i-1}m_r$, and $b(\mathfrak w_{i,j}'(\pi))=b(\pi)-(i-1)-\sum_{r=1}^{i-1}m_r-(j-1)$. It is now straightforward to check that for every integer $p$ with $3-b(\pi)\leq p\leq 2$, there is a unique operation that generates from $\pi$ a permutation $\widetilde\pi\in\mathcal U_{2k+3}(312,1342)$ with $b(\widetilde\pi)=b(\pi)+p$. In summary, a generating tree of the combinatorial class $\mathcal U(312,1342)$ is
\begin{equation}\label{Eq14}
\text{Axiom: }(3)\qquad\text{Rule: }(m)\leadsto(3)(4)\cdots(m+2)\quad\text{for every }m\in\mathbb N.
\end{equation}

Of course, \eqref{Eq13} and \eqref{Eq14} are identical. Thus, there is a natural isomorphism\footnote{We haven't formally defined ``isomorphisms" of generating trees, but we expect the notion will be apparent.} between the generating trees of intervals in noncrossing partition lattices and uniquely sorted permutations avoiding $312$ and $1342$. In fact, this isomorphism is unique. Finally, we obtain the bijections $\Upsilon_k:\mathcal U_{2k+1}(312,1342)\to\Int(\NC_k)$ from this isomorphism of generating trees in the obvious fashion, proving Theorem \ref{Thm5}. Throughout this section, we have chosen our notation in an attempt to make certain aspects of this bijection apparent. For example, if $\Upsilon_k(\pi)=(\rho,\kappa)$, then the nonconjoined hooks in the skyline of $\pi$ correspond to the exposed blocks in $\kappa$. Similarly, the hooks in the skyline of $\pi$ correspond to the exposed blocks in $\rho$ that are contained in exposed blocks of $\kappa$. 

Using the equation \eqref{Eq3}, we obtain the following corollary. 

\begin{corollary}\label{Cor3}
For each nonnegative integer $k$, \[|\mathcal U_{2k+1}(312,1342)|=\frac{1}{2k+1}{3k\choose k}.\]
\end{corollary}

\section{Pallo Comb Intervals and $\mathcal U_{2k+1}(231,4132)$}\label{Sec:Pallo}
Aval and Chapoton showed how to decompose the intervals in Pallo comb posets in order to obtain the identity \eqref{Eq15}. In this section, we show how to decompose permutations in $\mathcal U_{2k+1}(231,4132)$ in order to obtain a similar identity that proves these permutations are in bijection with Pallo comb intervals. 

\begin{theorem}\label{Thm6}
We have \[\sum_{k\geq 0}|\mathcal U_{2k+1}(231,4132)|x^k=C(xC(x)),\] where $C(x)=\dfrac{1-\sqrt{1-4x}}{2x}$ is the generating function of the sequence of Catalan numbers. 
\end{theorem} 

We know that the Tamari lattice $\mathcal L_k^T$ is an extension of the Pallo comb poset $\PC_k$. If we combine Theorems \ref{Thm3} and \ref{Thm4}, we find a bijection $\DL_k\circ\swl\circ\swu$ from $\mathcal U_{2k+1}(231,4132)$ to a subset of $\Int(\mathcal L_k^T)$. One might hope to prove Theorem \ref{Thm6} by showing that this subset is precisely $\Int(\PC_k)$ and then invoking \eqref{Eq15}. Unfortunately, this is not the case when $k=3$ (and probably also when $k\geq 4$). For example, $2154367\in\mathcal U_7(231,4132)$, but $\DL_3\circ\swl\circ\swu(2154367)=(UUDDUD,UUUDDD)\not\in\Int(\PC_3)$ (see Figure \ref{Fig2}). Before we can prove Theorem \ref{Thm6}, we need the following lemma.

\begin{lemma}\label{Lem9}
For each nonnegative integer $k$, $|\mathcal U_{2k+1}(132,231)|=C_k$. 
\end{lemma}
\begin{proof}\hspace{-.15cm}\footnote{One could alternatively prove this lemma by showing that $\DL_k\circ\swl:\mathcal U_{2k+1}(132,231)\to\Int(\mathcal A_k)$ is a bijection.}
It is known that $|\Av_{2k+1}(132,231)|=2^{2k}$. One way to prove this is to first observe that a permutation avoids $132$ and $231$ if and only if it can be written as a decreasing sequence followed by an increasing sequence. Given $\pi\in\Av_{2k+1}(132,231)$, we can write $\pi=L1R$, where $L$ is decreasing and $R$ is increasing. Let $w_\ell=U$ if $2k+2-\ell$ is an entry in $R$, and let $w_\ell=D$ if $2k+2-\ell$ is an entry in $L$. We obtain a word $w=w_1\cdots w_{2k}\in\{U,D\}^{2k}$. The map $\pi\mapsto w$ is a bijection between $\Av_{2k+1}(132,231)$ and $\{U,D\}^{2k}$. The permutation $\pi$ has exactly $k$ descents if and only if the letter $D$ appears exactly $k$ times in the corresponding word $w$. Furthermore, $\pi$ has a canonical hook configuration (meaning it is sorted) if and only if every prefix of $w$ contains at least as many occurrences of the letter $U$ as occurrences of $D$. Using Theorem \ref{Thm1}, we see that $\pi$ is uniquely sorted if and only if $w$ is a Dyck path. 
\end{proof}

\begin{proof}[Proof of Theorem \ref{Thm6}]
Let $B(x)=\sum_{k\geq 0}|\mathcal U_{2k+1}(231,4132)|x^k$ and $\widetilde B(x)=\sum_{n\geq 1}|\mathcal U_n(231,4132)|x^n$. Since there are no uniquely sorted permutations of even length, we have $\widetilde B(x)=xB(x^2)$. Therefore, our goal is to show that $\widetilde B(x)=E(x)$, where $E(x)=xC(x^2C(x^2))$. Using the standard Catalan functional equation $C(x)=1+xC(x)^2$, we find that $E(x)=x+xC(x^2)E(x)^2$. This last equation and the condition $E(x)=x+O(x^2)$ uniquely determine the power series $E(x)$. Since $\widetilde B(x)=x+O(x^2)$, we are left to prove that 
\begin{equation}\label{Eq16}
\widetilde B(x)=x+xC(x^2)\widetilde B(x)^2.
\end{equation}

The term $x$ in \eqref{Eq16} represents the permutation $1$. Now suppose $\pi\in\mathcal U_n(231,4132)$, where $n=2k+1\geq 3$. Proposition \ref{Prop1} tells us that $\pi$ has a canonical hook configuration $\mathcal H$, and Lemma~\ref{Lem2} tells us that the point $(2k+1,2k+1)$ is the northeast endpoint of a hook $H$ in $\mathcal H$. Let $(i,\pi_i)$ be the southwest endpoint of $H$. We will say $\pi$ is \emph{nice} if $i=1$. Let us first consider the case in which $\pi$ is nice. 

Because $\pi$ avoids $231$, we can write $\pi=\pi_1\,\lambda\,\mu\,(2k+1)$, where $\lambda\in S_{\pi_1-1}$ and $\mu$ is a permutation of $\{\pi_1+1,\ldots,2k\}$. Since $\pi$ avoids $231$ and $4132$, $\lambda$ avoids $132$ and $231$. As mentioned in the proof of Lemma \ref{Lem9}, this forces $\lambda$ to be a decreasing sequence followed by an increasing sequence. Let $m$ be the largest integer such that the subpermutation $\tau$ of $\lambda$ formed by the entries $1,2,\ldots,2m+1$ is in $\mathcal U_{2m+1}(132,231)$. We can write $\lambda=L\,\tau\, R$, where $\tau\in \mathcal U_{2m+1}(132,231)$, $L$ is decreasing, and $R$ is increasing. Consider the point $\mathfrak Q$ in the plot of $\pi$ whose height is the first entry in $R$. This point is not a descent bottom, so it follows from Lemma~\ref{Lem2} that it is the northeast endpoint of a hook $\mathcal H$. All of the hooks with southwest endpoints in $\tau$ have their northeast endpoints in $\tau$ because $\tau$ is uniquely sorted (restricting $\mathcal H$ to the plot of $\tau$ yields the canonical hook configuration of $\tau$). This means that the hook whose northeast endpoint is $\mathfrak Q$ has its southwest endpoint in the plot of $L$. Since $\mathfrak Q$ is also the lowest point in the plot of $R$, this shows that the smallest entry in $L$ is smaller than the smallest entry in $R$.

We claim that the permutation
$\pi'=\pi_1\,L\,R\,\mu$ is a uniquely sorted permutation that avoids $231$ and $4132$. Because the smallest entry in $L$ is smaller than the smallest entry in $R$, it is straightforward to check that $\pi'$ is a permutation of length $2k-2m-1$ that avoids $231$ and $4132$ and has exactly $k-m-1$ descents; we need to show that it has a canonical hook configuration $\mathcal H'$. We obtain $\mathcal H'$ from the canonical hook configuration $\mathcal H=(H_1,\ldots,H_k)$ of $\pi$ as follows. Let $\ell$ be the length of $L$. For all $r\in\{1,\ldots,\ell+m+1\}$, the southwest endpoint of $H_r$ is $(r,\pi_r)$. For $r\in\{1,\ldots,\ell\}$, let $H_r'$ be the hook of $\pi'$ with southwest endpoint $(r,\pi_r)=(r,\pi_r')$ whose northeast endpoint has the same height as the northeast endpoint of $H_{r+1}$. For $r\in\{\ell+m+2,\ldots,k\}$, let $H_r'$ be the hook of $\pi'$ whose southwest and northeast endpoints have the same heights as the southwest and northeast endpoints of $H_r$, respectively. The canonical hook configuration of $\pi'$ is $\mathcal H'=(H_1',\ldots,H_\ell',H_{\ell+m+2}',\ldots,H_k')$. See Figure \ref{Fig12} for an example of this construction. 

\begin{figure}[h]
\begin{center}
\includegraphics[width=.6\linewidth]{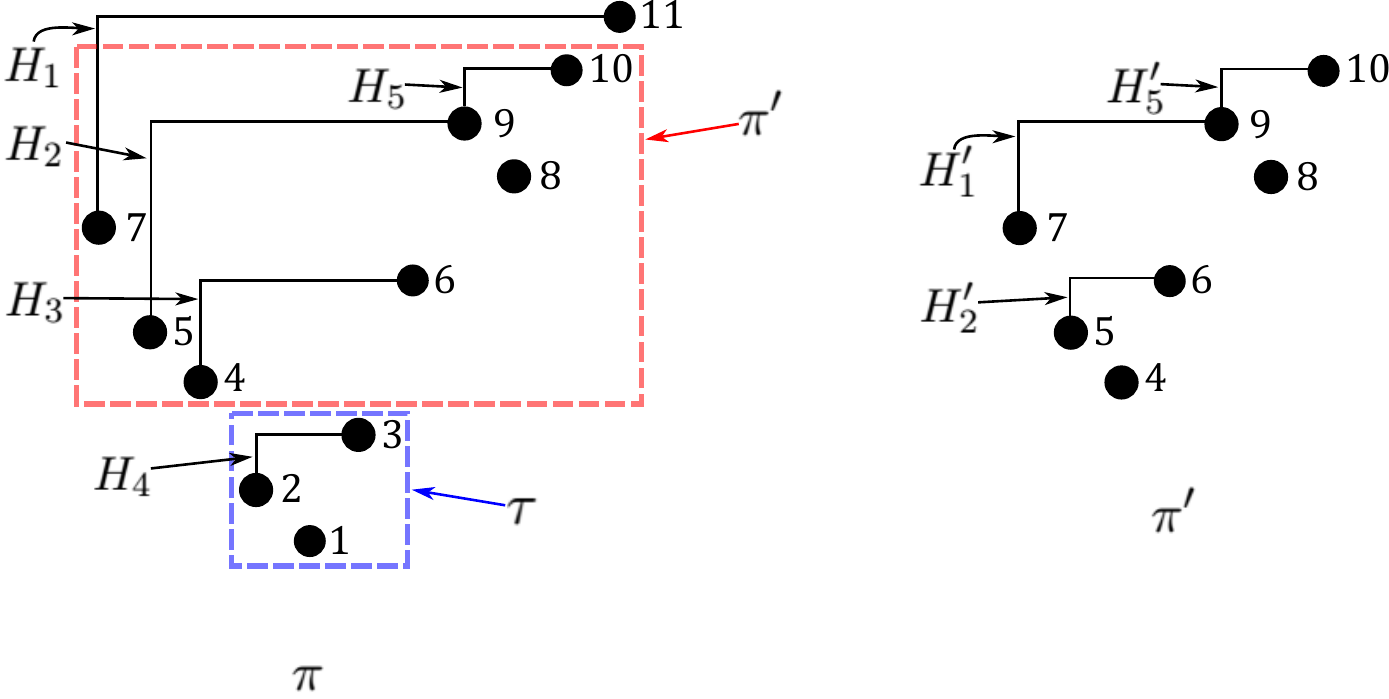}
\caption{Decomposing the nice permutation $\pi$ into pieces. In this example, $k=5$, $\ell=2$, and $m=1$.}
\label{Fig12}
\end{center}  
\end{figure}

Because the last (and smallest) entry in $L$ is smaller than the first (and smallest) entry in $R$, the last entry in $L$ is $2m+2$. This is also the smallest entry in $\pi'$. Letting $\pi''$ denote the normalization of $\pi'$, we have obtained from $\pi$ the pair $\alpha(\pi):=(\pi'',\tau)\in \mathcal U_{2k-2m-1}(231,4132)\times\mathcal U_{2m+1}(132,231)$. 

We can reverse this procedure. If we are given $(\pi'',\tau)\in \mathcal U_{2k-2m-1}(231,4132)\times\mathcal U_{2m+1}(132,231)$, then we can increase all of the entries in $\pi''$ by $2m+1$ to form $\pi'$. We then insert $\tau$ after the smallest entry in $\pi'$ and append the entry $2k+1$ to the end to form the permutation $\pi$. 
This construction cannot produce a new $231$ pattern or a  new $4132$ pattern, so the resulting permutation $\pi$ avoids these patterns. The resulting permutation also has $k$ descents. The canonical hook configuration of $\pi$ is obtained from those of $\pi''$ and $\tau$ by reversing the above procedure. Furthermore, this construction forces the leftmost point in the plot of $\pi$ to be the southwest endpoint of a hook with northeast endpoint $(2k+1,2k+1)$. Thus, the permutation $\pi$ obtained by combining $\pi''$ and $\tau$ is nice and is in $\mathcal U_{2k+1}(231,4132)$. We need to show that $\alpha(\pi)=(\pi'',\tau)$. To this end, choose some $m'>m$ such that $2m'+1<\pi_1$. Because $\pi$ avoids $231$ and $4132$, the subpermutation of $\pi$ consisting of the entries $1,\ldots,2m'+1$ is a decreasing sequence followed by an increasing sequence. Let us write it as the concatenation $\widehat L\,(2m+2)\,\widetilde L\,1\,\widetilde R\,\widehat R$, where $\widetilde L\,1\,\widetilde R=\tau$. The entries in $\widetilde L$ correspond to the southwest endpoints of the hooks in the canonical hook configuration of $\tau$, and the northeast endpoints of these hooks correspond to the entries in $\widetilde R$. The entries in $\widehat L$ correspond to (some of the) southwest endpoints of the hooks in the canonical hook configuration of $\pi'$, and the entries in $\widetilde R$ correspond to a subset of the northeast endpoints of these hooks. Thus, $|\widetilde L|=|\widetilde R|$ and $|\widehat L|\geq |\widehat R|$. It follows that the permutation $\widehat L\,(2m+2)\,\widetilde L\, 1\,\widetilde R\,\widehat R$, which consists of the entries $1,\ldots,2m'+1$, has more than $m'$ descents; thus, it is not uniquely sorted. This means that $m$ is in fact the largest integer with $2m+1<\pi_1$ such that the subpermutation of $\pi$ consisting of the entries $1,\ldots,2m+1$ is uniquely sorted. This subpermutation is precisely $\tau$, so $\alpha(\pi)=(\pi'',\tau)$. Lemma \ref{Lem9} tells us that $\sum_{n\geq 1}|\mathcal U_n(132,231)|x^n=xC(x^2)$, so it follows that the generating function that counts nice permutations in $\mathcal U(231,4132)$ is $x^2C(x^2)\widetilde B(x)$.

We now consider the general case in which $\pi$ is not necessarily nice (see the left side of Figure~\ref{Fig13} for an example). Let $\sigma=\pi_1\cdots\pi_{i-1}$, and let $\sigma'$ be the normalization of $\pi_i\cdots\pi_{2k+1}$. Let $\sigma''$ be the normalization of $\sigma\pi_i$. If there were an index $\delta\in\{1,\ldots,i-1\}$ such that $\pi_\delta>\pi_i$, then we could choose this $\delta$ maximally and see that $(\delta,\pi_\delta)$ is a descent top of the plot of $\pi$. The point $(\delta,\pi_{\delta})$ would be the southwest endpoint of a hook, and this hook would have to lie above $H$, the hook with southwest endpoint $(i,\pi_i)$. However, the northeast endpoint of $H$ is $(2k+1,2k+1)$, so this is impossible. Therefore, every entry in $\sigma$ is smaller than $\pi_i$. Because $\pi$ avoids $231$, $\pi=\sigma\oplus\sigma'$ (so $\sigma''=\sigma i$). The permutation $\sigma''$ is in $\mathcal U_{i}(231,4132)$. The permutation $\sigma'$ is a nice permutation in $\mathcal U(231,4132)$. If we were given the permutation $\sigma''\in\mathcal U(231,4132)$ and the nice permutation $\sigma'\in\mathcal U(231,4132)$, then we could easily reobtain $\pi$ by first deleting the last entry of $\sigma''$ to form $\sigma$ and then writing $\pi=\sigma\oplus\sigma'$. It follows that $\widetilde B(x)-x=\frac{1}{x}(x^2C(x^2)\widetilde B(x))\widetilde B(x)=xC(x^2)\widetilde B(x)^2$, which is \eqref{Eq16} (the $\frac 1x$ comes from the fact that $\pi_1\cdots\pi_i$ and $\pi_i\cdots\pi_{2k+1}$ overlap in the entry $\pi_i$). 
\end{proof}

\begin{figure}[h]
\begin{center}
\includegraphics[width=.9\linewidth]{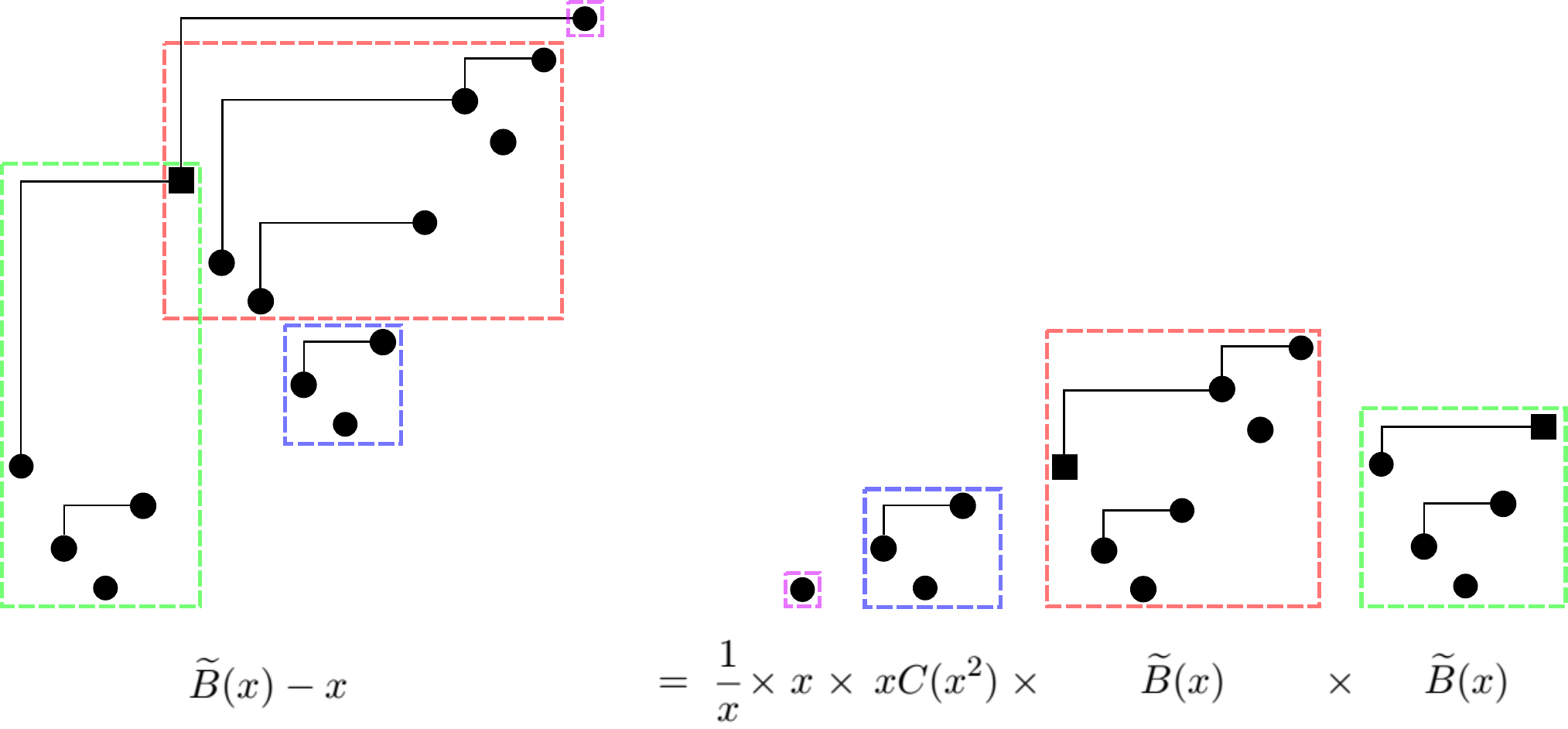}
\caption{An illustration of the equation $\widetilde B(x)-x=xC(x^2)\widetilde B(x)^2$ from the proof of Theorem \ref{Thm6}. The factor $\frac 1x$ comes from the fact that the point marked with the square appears twice on the right-hand side.}
\label{Fig13}
\end{center}  
\end{figure}

\section{Catalan Antichain Intervals}\label{Sec:Antichain}
In this section, we prove that \[|\mathcal U_{2k+1}(321)|=|\mathcal U_{2k+1}(231,312)|=|\mathcal U_{2k+1}(132,231)|=|\mathcal U_{2k+1}(132,312)|=C_k.\] These results fit into the theme of this article if we interpret $C_k$ as the number of intervals in the antichain $\mathcal A_k$. 

\begin{theorem}\label{Thm8}
For each nonnegative integer $k$, we have $|\mathcal U_{2k+1}(321)|=C_k$. 
\end{theorem}
\begin{proof}
A \emph{nondecreasing parking function of length $k$} is a tuple $(a_1,\ldots,a_k)$ of positive integers such that $a_1\leq\cdots\leq a_k$ and $a_i\leq i$ for all $i\in[k]$. It is well known that the number of nondecreasing parking functions of length $k$ is $C_k$. Given $\pi=\pi_1\cdots\pi_{2k+1}\in\mathcal U_{2k+1}(321)$, put $a_i=\pi_{2i}-i+1$. We claim that $(a_1,\ldots,a_k)$ is a nondecreasing parking function.

Note that $(2,\pi_2)$ is not the northeast endpoint of a hook in the canonical hook configuration of $\pi$. Lemma \ref{Lem2} tells us that $(2,\pi_2)$ is a descent bottom in the plot of $\pi$, so $1$ is a descent of $\pi$. Because $\pi$ is sorted, $\pi_{2k+1}=2k+1$. This implies that $2k$ is not a descent of $\pi$. Since $\pi$ avoids $321$, no two descents of $\pi$ are consecutive integers. We know by Theorem \ref{Thm1} that $\pi$ has $k$ descents, so these descents must be $1,3,5,\ldots,2k-1$. Choose $i\in[k]$. Since $\pi$ avoids $321$ and $\pi_{2i-1}>\pi_{2i}$, all of the elements of $[\pi_{2i}-1]$ appear to the left of $\pi_{2i}$ in $\pi$. Because $\pi_{2i-1}$ is an additional entry that appears to the left of $\pi_{2i}$ in $\pi$, we must have $2i-1\geq\pi_{2i}$. It follows that $a_i=\pi_{2i}-i+1\leq i$. If $i\in[k-1]$, then $\pi_{2i+2}\geq\pi_{2i}+1$ since $\pi$ avoids $321$ and $\pi_{2i-1}>\pi_{2i}$. This means that $a_{i+1}=\pi_{2i+2}-(i+1)+1\geq\pi_{2i}-i+1=a_i$.  As $i$ was arbitrary, $(a_1,\ldots,a_k)$ is a nondecreasing parking function. Notice also that if we construct the canonical hook configuration of $\pi$ using the construction described in Section \ref{Sec:StackBack}, we are forced to choose $(2i+1,\pi_{2i+1})$ as the northeast endpoint of the hook whose southwest endpoint is $(2i-1,\pi_{2i-1})$. This implies that $\pi_1<\pi_3<\cdots<\pi_{2k+1}$.

Given the nondecreasing parking function $(a_1,\ldots,a_k)$, we can reobtain the permutation $\pi$. Indeed, the values of $\pi_2,\pi_4,\ldots,\pi_{2k}$ are determined by the definition $a_i=\pi_{2i}-i+1$. The other entries of $\pi$ are determined by the fact that $\pi_1<\pi_3<\cdots<\pi_{2k+1}$. This is because $\pi_{2i-1}$ must be the $i^\text{th}$-smallest element of $[2k+1]\setminus\{\pi_2,\pi_4,\ldots,\pi_{2k}\}$. We want to check that the permutation $\pi$ obtained in this way is indeed in $\mathcal U_{2k+1}(321)$. One can easily check that this permutation avoids $321$ and has $k$ descents. We must show that it has a canonical hook configuration $\mathcal H=(H_1,\ldots,H_k)$. This is easy; $H_i$ is simply the hook with southwest endpoint $(2i-1,\pi_{2i-1})$ and northeast endpoint $(2i+1,\pi_{2i+1})$.    
\end{proof}

In the following theorems, recall the bijection $\DL_k:\mathcal U_{2k+1}(312)\to\Int(\mathcal L_k^S)$ from Theorem \ref{Thm2}. 

\begin{theorem}\label{Thm9}
For each nonnegative integer $k$, the restriction of $\DL_k$ to $\mathcal U_{2k+1}(231,312)$ is a bijection from $\mathcal U_{2k+1}(231,312)$ to $\Int(\mathcal A_k)$. Hence, $|\mathcal U_{2k+1}(231,312)|=C_k$.
\end{theorem}
\begin{proof}
A permutation is called \emph{layered} if can be written as $\delta_{a_1}\oplus\delta_{a_2}\oplus\cdots\oplus\delta_{a_m}$ for some positive integers $a_1,\ldots,a_m$, where $\delta_a=a(a-1)\cdots 1$ is the decreasing permutation in $S_a$. It is a standard fact among permutation pattern enthusiasts that a permutation $\pi\in S_n$ is layered if and only if it avoids $231$ and $312$. It is straightforward to check that a permutation $\pi\in\mathcal U_{2k+1}(312)$ is layered if and only if $\DL_k(\pi)\in\Int(\mathcal A_k)$ (meaning $\DL_k(\pi)=(\Lambda,\Lambda)$ for some $\Lambda\in{\bf D}_k$). 
\end{proof}

\begin{theorem}\label{Thm7}
For each nonnegative integer $k$, the restriction of $\DL_k\circ\swl$ to $\mathcal U_{2k+1}(132,231)$ is a bijection from $\mathcal U_{2k+1}(132,231)$ to $\Int(\mathcal A_k)$. Hence, $|\mathcal U_{2k+1}(132,231)|=C_k$.
\end{theorem}
\begin{proof}
We already know from Lemma \ref{Lem9} that $|\mathcal U_{2k+1}(132,231)|=C_k$. Now suppose $\pi\in$ \linebreak $\mathcal U_{2k+1}(132,231)$. We know by Theorem \ref{Thm1} that $\pi$ is sorted and has $k$ descents. Lemmas \ref{Lem1} and \ref{Lem6} tell us that $\swl(\pi)$ is sorted and has $k$ descents, so it follows from Theorem \ref{Thm1} that $\swl(\pi)$ is uniquely sorted. We know from Lemma \ref{Lem10} that $\swl(\pi)\in\Av(231,312)$. Lemma \ref{Lem10} also tells us that $\swl$ is injective on $\mathcal U_{2k+1}(132,231)$. We have proven that $\swl:\mathcal U_{2k+1}(132,231)\to\mathcal U_{2k+1}(231,312)$ is an injection. It must also be surjective because $|\mathcal U_{2k+1}(132,231)|=|\mathcal U_{2k+1}(231,312)|=C_k$ by Lemma \ref{Lem9} and Theorem \ref{Thm9}. The proof of the theorem now follows from Theorem \ref{Thm9}. 
\end{proof}

\begin{theorem}\label{Thm10}
For each nonnegative integer $k$, the restriction of $\DL_k\circ\swd$ to $\mathcal U_{2k+1}(132,312)$ is a bijection from $\mathcal U_{2k+1}(132,312)$ to $\Int(\mathcal A_k)$. Hence, $|\mathcal U_{2k+1}(132,312)|=C_k$.
\end{theorem}
\begin{proof}
Theorem \ref{Thm3} and Lemma \ref{Lem10} tell us that $\swd:\mathcal U_{2k+1}(132)\to\mathcal U_{2k+1}(231)$ and $\swd:\Av(132,312)\to\Av(231,312)$ are bijections. It follows that $\swd:\mathcal U_{2k+1}(132,312)\to\mathcal U_{2k+1}(231,312)$ is a bijection, so the proof of the theorem follows from Theorem \ref{Thm9}. 
\end{proof}

\section{Concluding Remarks}\label{Sec:Conclusion}

One of our primary focuses in this paper has been the enumeration of sets of the form \linebreak $\mathcal U_{2k+1}(\tau^{(1)},\ldots,\tau^{(r)})$. We can actually complete this enumeration for all possible cases in which the patterns $\tau^{(1)},\ldots,\tau^{(r)}$ are of length $3$. It is easy to check that $\mathcal U_{2k+1}(123)$ and $\mathcal U_{2k+1}(213)$ are empty when $k\geq 2$, so we only need to focus on the cases in which $\{\tau^{(1)},\ldots,\tau^{(r)}\}\subseteq\{132,231,312,321\}$. When $r=0$ (so we consider the set $\mathcal U_{2k+1}$), the enumeration is completed in \cite{DefantEngenMiller} and is given by Lassalle's sequence. The cases in which $r=1$ are handled in Corollary \ref{Cor1}, Corollary \ref{Cor2}, and Theorem \ref{Thm8}. Three of the six cases in which $r=2$ are handled in Theorems \ref{Thm9}, \ref{Thm7}, and \ref{Thm10}. In the following theorem, we finish the other three cases in which $r=2$ along with all of the cases in which $r=3$ or $r=4$. 

\begin{theorem}\label{Thm11}
For each nonnegative integer $k$, we have \[|\mathcal U_{2k+1}(231,321)|=|\mathcal U_{2k+1}(312,321)|=|\mathcal U_{2k+1}(231,312,321)|=|\mathcal U_{2k+1}(132,231,312)|=1.\] For each $k\geq 2$, we have $\mathcal U_{2k+1}(132,321)=\emptyset$.
\end{theorem}
\begin{proof}
We may assume $k\geq 2$. The proof of Theorem \ref{Thm8} shows that if $\pi=\pi_1\cdots\pi_{2k+1}\in\mathcal U_{2k+1}(321)$, then $\pi_1<\pi_3<\cdots<\pi_{2k+1}$, and the descents of $\pi$ are $1,3,5,\ldots,2k-1$. It easily follows that \[\mathcal U_{2k+1}(231,321)=\mathcal U_{2k+1}(312,321)=\mathcal U_{2k+1}(231,312,321)=\{214365\cdots(2k)(2k-1)(2k+1)\}\] and that $\mathcal U_{2k+1}(132,321)=\emptyset$. Every element of $\Av(132,231,312)$ is of the form $L\oplus R$, where $L$ is a decreasing permutation and $R$ is an increasing permutation. A uniquely sorted permutation of length $2k+1$ must have $k$ descents, so \[\mathcal U_{2k+1}(132,231,312)=\{(k+1)k\cdots 1(k+2)(k+3)\cdots(2k+1)\}. \qedhere\] 
\end{proof}

Theorem \ref{Thm11} implies that \[\mathcal U_{2k+1}(132,231,321)=\mathcal U_{2k+1}(132,312,321)=\mathcal U_{2k+1}(132,231,312,321)=\emptyset,\] so we have completed the enumeration of $\mathcal U_{2k+1}(\tau^{(1)},\ldots,\tau^{(r)})$ when $\tau^{(1)},\ldots,\tau^{(r)}$ are of length $3$. Since we enumerated $\mathcal U_{2k+1}(312,1342)$ in Corollary \ref{Cor3} and enumerated $\mathcal U_{2k+1}(231,4132)$ in Theorem \ref{Thm6}, it is natural to look at other sets of the form $\mathcal U_{2k+1}(\tau^{(1)},\tau^{(2)})$ with $\tau^{(1)}\in S_3$ and $\tau^{(2)}\in S_4$. To this end, we have eighteen conjectures. Each row of the following table represents the conjecture that the class of (normalized) uniquely sorted permutations (of odd length) avoiding the given patterns is counted by the corresponding OEIS sequence.

{\tabulinesep=1.2mm
\begin{table}[h]

\begin{tabu}{|c|[1.5pt]c|}

    \hline Patterns & Sequence \\\tabucline[2.5pt]{-}
    ${\dagger}$ $312,1432$ &  \\
    ${\dagger}$ $312,2431$ &  \\
    ${\dagger}$ $312,3421$ & A001764 \\
    ${\dagger}$ $132,3412$ &  \\
    ${\dagger}$ $231,1423$ &  \\\hline
    \hphantom{$\dagger$} $312,1243$ &  A122368 \\\hline
\end{tabu}
\quad
\begin{tabu}{|c|[1.5pt]c|}

    \hline Patterns & Sequence \\\tabucline[2.5pt]{-}
    ${\dagger}$ $132,3421$ &  \\
    ${\dagger}$ $132,4312$ & A001700 \\
    \hphantom{$\dagger$} $231,1243$ &  \\\hline \noalign{\vskip -.08cm}  
    \makecell{\,\,\hspace{.3cm}$132,2341$ \vspace{.05cm}\\ \hspace{.4cm}$132,4123$ \hspace{-.3cm}} & \hspace{.13cm}A109081 \vspace{-.155cm} \\\hline
    \hphantom{$\dagger$} $312,2341$ & A006605 \\\hline
\end{tabu}
\quad
\begin{tabu}{|c|[1.5pt]c|}

    \hline Patterns & Sequence \\\tabucline[2.5pt]{-}
    \hphantom{$\dagger$} $312,3241$ & A279569 \\\hline
    \hphantom{$\dagger$} $312,4321$ & A063020 \\\hline
    \hphantom{$\dagger$} $132,4231$ & A071725 \\\hline
    ${\dagger}$ $231,1432$ & A001003 \\\hline
    ${\dagger}$ $231,4312$ & A127632 \\\hline
    \hphantom{$\dagger$} $231,4321$ & A056010 \\\hline
\end{tabu}\vspace{.3cm}
\caption{Conjectural OEIS sequences enumerating some sets of the form $\mathcal U_{2k+1}(\tau^{(1)},\tau^{(2)})$. Mularczyk recently proved the 9 conjectures that are marked with the symbol $\dagger$.}\label{Tab1}
\end{table}}

Note that the OEIS sequences A001764 and A127632 give the numbers appearing in \eqref{Eq3} and \eqref{Eq15}, respectively. A couple of especially well-known sequences appearing in Table \ref{Tab1} are A001700, which consists of the binomial coefficients ${2k-1\choose k}$, and A001003, which consists of the little Schr\"oder numbers. In the time since the original preprint of this article was released online, Hanna Mularczyk \cite{Hanna} proved half of these 18 conjectures. She used a mixture of generating function arguments and interesting bijections that link pattern-avoiding uniquely sorted permutations with Dyck paths, ${\bf S}$-Motzkin paths, and Schr\"oder paths. We have marked the conjectures that she settled with the symbol $\dagger$ in Table \ref{Tab1}.

We have also calculated the first few values of $|\mathcal U_{2k+1}(231,4123)|$; beginning at $k=0$, they are $1,1,3,10,36,138,553,2288,9699,41908$. This sequence appears to be new, so we have added it as sequence A307346 in the OEIS. We have also computed the first few terms in each of the $24$ sequences $(|\mathcal U_{2k+1}(\tau)|)_{k\geq 0}$ for $\tau\in S_4$; most of these sequences appear to be new. 

\section{Acknowledgments}
The author thanks Niven Achenjang, who helped with the production of data that led to some of the conjectures in Section \ref{Sec:Conclusion}. He also thanks William Kuszmaul, who used his recently-developed fast algorithm \cite{Kuszmaul} for generating permutations avoiding certain patterns. This produced data that was used in formulating two of the conjectures in Section \ref{Sec:Conclusion}. The author is gracious to the two referees whose helpful comments vastly improved the presentation of this article. The author was supported by a Fannie and John Hertz Foundation Fellowship and an NSF Graduate Research Fellowship.

\end{document}